\documentclass[11pt,oneside]{amsart}
\usepackage{geometry}                % See geometry.pdf to learn the layout options. There are lots.
\usepackage{graphicx}
\usepackage{amsmath}
\usepackage{amsfonts}
\usepackage{amssymb}
\usepackage{enumitem}
\usepackage{epstopdf}
\usepackage{caption}
\usepackage{latexsym}
\usepackage{tikz}

\newtheorem{theorem}{Theorem}
\newtheorem{proposition}[theorem]{Proposition}
\newtheorem{lemma}[theorem]{Lemma}
\newtheorem{corollary}[theorem]{Corollary}
\newtheorem{definition}[theorem]{Definition}
\newtheorem{claim}[theorem]{Claim}
\newtheorem{observation}[theorem]{Observation}

\newtheorem{conjecture}[theorem]{Conjecture}

\numberwithin{theorem}{section}

\def\TASK#1{}
\def\M{\mathcal M}

\def\F{\mathcal F}

\newcommand\Bbox{
{%\unskip
\nobreak\hfil\penalty50
\hskip1em\hbox{}\nobreak\hfil{\lower .5pt \hbox{$\Box$}}
\parfillskip=0pt \finalhyphendemerits=0 \par}
}
\newcommand\eop{\ifmmode {\hbox{\Bbox}} \else \Bbox \fi}

\newcommand\dist{\mathrm{dist}}

\title{Embedding graphs having Ore-degree at most five}

\author{B\'ela Csaba}
\address{Bolyai Institute, University of Szeged, Hungary}
\email{bcsaba@math.u-szeged.hu}

\thanks{The first author was partially supported by ERC-AdG. 321104 and by the National Research, Development and Innovation Office - NKFIH Funds No.~SNN-117879 and KH\_18 129597. The second author was 
 partially supported by NKFIH Fund No.~KH\_18 129597.}

\author{Judit Nagy-Gy\"orgy }\address{Bolyai Institute, University of Szeged, Hungary}
\email{ngyj@math.u-szeged.hu}

\begin{document}

\maketitle

\begin{abstract}
%\noindent
Let $H$ and $G$ be graphs on $n$ vertices, where $n$ is sufficiently large. We prove that if $H$ has Ore-degree at most 5 and $G$ has minimum degree at least $2n/3$ then $H$ is a subgraph of $G.$
\end{abstract}

%\begin{keywords}
%embedding problems, Ore-degree, Bollob\'as-Eldridge-Catlin conjecture
%\end{keywords}

% REQUIRED
%\begin{AMS}
%	05C35, 05C75, 05C07, 68R10
%\end{AMS}

%\frenchspacing

\section{Notations}

In this paper we will consider only simple graphs. We mostly use standard graph theory notation: $V(G)$ is the vertex set 
and $E(G)$ is the edge set of graph $G$, $v(G)=|V(G)|$, $e(G)=|E(G)|$, $\deg_G(x)$ (or $\deg(v)$ if it is unambigous) is 
the degree of vertex $x\in V(G)$, $\delta(G)$ the minimum and $\Delta(G)$ the maximum degree of $G$. 
We write $N_G(x)$ (or $N(x)$) for the neighborhood of the vertex $x\in V(G)$, and $N(X)=\bigcup_{x\in X}N(x)$. 
Let $N_G(x,A)=N_G(x)\cap A$ and denote $\deg_G(x,A)$ the number of neighbors of $x$ in the set $A$. 
Let $N_G(X,A)=\bigcup_{x\in X}N_G(x,A)$. 
%The $\dist_G(x,y)$ distance of vertices $x,y\in V(G)$ is the length of a shortest $xy$ path in $G$. If there is no $xy$ path 
%then it is infinite by definition.

%If $A$ and $B$ are disjoint subsets of $V(G)$ then we denote by $e(A,B)$ the number of edges with one endpoint in
%$A$ and the other in $B$. The density between disjoint sets $X$ and $Y$ is defined as
%$$d(X,Y)=\frac{e(X,Y)}{|X|\cdot|Y|}.$$

We let $\overline{G}$ denote the complement of $G$ where $V(\overline{G})=V(G)$ and $E(\overline{G})=\binom{V(G)}{2}-E(G)$.
If $A\subset V(G)$ we write $G-A$ for the graph induced by the vertices of $V(G)-A$. Moreover $G[A]$ is a shorthand for $G-(V(G)-A)$. Given two disjoint sets, $A, B\subset V(G)$ we write $G[A, B]$ for the bipartite subgraph of $G$ which contains the edges
that connect a vertex of $A$ with a vertex of $B.$

If $G$ has a subgraph isomorphic to $H$ then we write $H\subset G$. In this case we call $G$ the
host graph. Given graphs $H$ and $G,$ the mapping $\phi: V(H)\to V(G)$ is a homomorphism, if $\phi(x)\phi(y)\in E(G)$ whenever 
$xy\in E(H).$ 
%In particular, $\phi^{-1}(v)$ is an independent set in $H$ for all $v\in V(G),$ since $G$ is a simple graph.

We say that $G$ has an $H$-factor if there are $\lfloor v(G)/v(H)\rfloor$ vertex-disjoint copies of $H$ in $G$ 
(this notion is somewhat different from the common one: we do not require that $v(G)$ is a multiple
of $v(H)$). Throughout the paper we will apply the relation ``$\ll$'': $a\ll b$ if $a$ is sufficiently smaller
than $b$.

\section{Introduction}

The fundamental question of extremal graph theory can be formulated as an embedding problem as follows: Given two graphs, $H$ and $G,$ under
what conditions is $H$ the subgraph of $G?$ Equivalently, one can consider embedding problems as packing problems as follows:
under what conditions can we find an edge disjoint copy of $H$ and $\overline{G}$ in $K_n?$ Sometimes it is more convenient to investigate
the equivalent packing version of an embedding question, we will use both notions in the paper.

For a fixed subgraph $H\subset G$ Tur\'an's theorem tells that $G$ has to be sufficiently dense in terms of the chromatic number of $H.$ If $G$ and
$H$ have the same order, the density of $G$ in general is not sufficient anymore, instead one needs e.g.~large minimum degree in $G.$ 
The most famous example of this type is Dirac's celebrated theorem on Hamilton cycles~\cite{D}: if $n\ge 3$, $G$ is a graph of order $n$ 
and $\delta(G)\ge n/2$ then $G$ has a Hamilton cycle. 

His theorem was generalized by Ore~\cite{Ore}: Assume that $n\ge 3,$ $G$ is a graph of order $n$ and for every $xy\not\in E(G)$ we have $deg(x)+deg(y)\ge n.$ Then
$G$ has a Hamilton cycle. This result motivates the notion of the {\it Ore-degree} of an edge~\cite{KKY, KY}: the Ore-degree of $xy$ is the 
sum $$\theta(x,y)=\deg(x)+\deg(y)$$ of the degrees of its endpoints. The Ore-degree of $G,$ denoted by $\theta(G),$ is 
the maximum Ore-degree in $E(G)$. The embedding problems that include minimum and maximum degree conditions 
are called {\it Dirac-type,} while those involving the Ore-degree are called {\it Ore-type} problems in the literature.

An excellent source of embedding results and conjectures is the survey~\cite{KKY} by Kierstead, Kostochka and Yu, in which several Ore-type questions are considered as well. One can easily formulate an Ore-type problem by replacing 
the maximum degree in a Dirac-type embedding problem by half of the Ore-degree of the graph, or if one considers a packing version, then one can replace even both maximum degrees. In some cases the resulting questions were solved.
For example, Kostochka and Yu proved \cite{KY2} that if $G_1$ and $G_2$ are graphs of order $n$ such that 
$\theta(G_1)\Delta(G_2)<n$ then $G_1$ and ${G_2}$ pack. This is in fact a {\it half-Ore} version of the famous packing result of 
Sauer and Spencer~\cite{SS}: if $\Delta(G_1)\cdot\Delta(G_2)<n/2$ then $G_1$ and $G_2$ pack. We remark that the {\it full-Ore} version
of the Sauer-Spencer theorem, when both maximum degrees are replaced by half of the corresponding Ore-degree, is open.
Another important Dirac-type theorem for which the half-Ore version was proved is by Aigner-Brandt~\cite{AB} and Alon-Fischer~\cite{AF} on embedding 2-factors. The half-Ore version was proved by Kostochka and Yu~\cite{KY3}.

One of the most important Dirac-type embedding problems was formulated by Bollob\'as and Eldridge~\cite{BE} and independently 
by Catlin~\cite{C}.

\begin{conjecture}[Bollob\'as, Eldridge; Catlin]
If $G_1$ and $G_2$ are graphs on $n$ vertices with maximum degree $\Delta_1$ and $\Delta_2,$ respectively, and 
$$(\Delta_1+1)(\Delta_2+1)\le n+1,$$ then $G_1$ and $G_2$ pack.  
\end{conjecture}

There are some resolved cases of the above conjecture, see e.g.,~\cite{AB, AF, Cs2, CsSSz}, but in general it is wide open. Kostochka and Yu~\cite{KY, KY3} conjectured a half-Ore version of it, in which $\Delta_1$ is
replaced by $\theta(G_1)/2.$ In this paper we suggest a slightly stronger conjecture\footnote{That is, instead of  $\theta(G_1)/2$
we have $\lfloor {\theta(G_1)}/{2}\rfloor.$ Hence, the two conjectures differ in case $\theta(G_1)$ is odd.}:

\begin{conjecture}\label{sejtes}
If $G_1$ and $G_2$ are graphs on $n$ vertices such that  
$$\left(\left\lfloor \frac{\theta(G_1)}{2}\right\rfloor +1\right)(\Delta_2+1)\le n+1,$$ then $G_1$ and $G_2$ pack.  
\end{conjecture}

In this paper we confirm the above conjecture for the cases $2\le \theta(G_1)\le 5$ (note that the Ore-degree is either 
zero or at least 2). The cases $\theta(G_1)=2, 3, 4$ are relatively easy. The case $\theta(G_1)=5$ is considerably harder. 
Note that there are expander graphs $H$ with $\theta(H)=5.$ We present our main result as an embedding question as follows.

\begin{theorem}\label{main}
There exists an $n_0$ such that if $n > n_0$, $\theta(H)\le 5$ for a graph $H$ of order $n$ and $\delta(G)\ge 2n/3$ for a 
graph $G$ of order $n$ then $H$ is a subgraph of $G.$
\end{theorem}

Let us call a graph $G$ with minimum degree $2n/3$ an {\it $\eta$-extremal} graph, if there exists $A\subset V(G)$ with 
$|A|=\lfloor n/3\rfloor$
such that $e(G[A])\le \eta {n/3 \choose 2}<\eta  n^2/18.$ If $G$ has no such subset we call it {\it $\eta$-non-extremal.}
We also prove the following {\it stability} version of Theorem~\ref{main}.

\begin{theorem}\label{stab}
There exist positive numbers  $\gamma_0, \eta_0, n_0$ such that if $n > n_0$, $\eta> \eta_0,$ $\gamma <\gamma_0$ and $\theta(H)\le 5$ for a 
graph $H$ of order $n$ and $\delta(G)\ge (2/3-\gamma)n$ for any $\eta$-non-extremal  
graph $G$ of order $n$ then $H$ is a subgraph of $G.$ 
\end{theorem}

%When proving Theorem~\ref{main} and Theorem~\ref{stab} we assume that $H$ is {\it saturated}, that is, if one adds any edge
%to the saturated $H$ then its Ore degree will grow above 5. This is easy to achieve: just add edges one-by-one to $H$ until it 
%becomes saturated (of course, it is not allowed to create parallel edges or loops when saturating $H$).

%\medskip

The structure of the paper is as follows. First we consider the cases $\theta(H)=2, 3, 4$ in a separate section. Then we provide a 
list of further notions and tools necessary for proving Theorem~\ref{main} and Theorem~\ref{stab}. Finally we turn to the proof of our main results.

\section{Graphs having Ore-degree at most four}
% -- Structure of graphs with small Ore-degree}

Let us begin with some simple observations. The proof of the claim below is left for the reader. 

\begin{claim}
Let $H$ be a non-empty graph. Then $\theta(H)\ge 2.$ If $\theta(H)=2$ then the connected components
of $H$ are isolated vertices and edges, and $H$ has at least one edge. 
If $\theta(H)=3$ then the components of $H$ are paths having length at most 2, and $H$ has at least one length-2 path.
\end{claim}

Assume that $H$ is a graph on $n$ vertices with $\theta(H)\le 3.$ It is easy to see that $H\subseteq C_n.$ 
Let $\delta_2(G)=\min_{xy\not\in E(G)} \{\deg(x)+\deg(y)\}.$ Applying the celebrated theorems of Dirac and Ore we have the following.

\begin{theorem}
Suppose that $G$ is a simple graph on $n\ge 3$ vertices and $H$ is a simple graph on $n$ vertices 
with $\theta(H)\le 3.$ If $\delta(G)\ge n/2$ or if $\delta_2(G)\ge n$ then $H\subseteq G$  
\end{theorem}

Of course, other sufficient conditions for the existence of Hamilton cycles (e.g.~by P\'osa or by Chv\'atal) also imply 
analogous results. The above theorem is tight as the following example shows.
Let $n=2k.$ Let $H$ be the vertex disjoint union of $k$ edges, and let $G$ be a complete bipartite graph with vertex 
class sizes $k+1$ and $k-1.$ Then $v(G)=v(H)=n,$ $\theta(H)=2,$ $\delta(G)=k-1=n/2-1,$ and 
$H\not\subset G.$

\medskip

Let us consider the case $\theta(H)=4$. The proof of the claim below is left for the reader.

\begin{claim}\label{theta4}
Assume that $H$ is a graph with $\theta(H)=4.$ Then the connected components of $H$ are paths, cycles
or claws (a claw is a $K_{1,3}$).
\end{claim}

One of the most important case is when $H$ contains vertex disjoint triangles. The following is a celebrated result of K.~Corr\'adi and A.~Hajnal:

\begin{theorem}[Corr\'adi--Hajnal, \cite{HC}]\label{CH}
If $G$ is a graph of order $n$ and $\delta(G)\ge 2n/3$ then $G$ has a $K_3$-factor.
\end{theorem}

Let us consider a complete 3-partite graph $G$ with vertex classes having cardinalities $k, k+1$ and $k-1,$ respectively. Clearly, $G$
does not contain $k$ vertex disjoint triangles, showing that the minimum degree bound of the Corr\'adi-Hajnal theorem is
tight. It also shows that $\delta(G)$ has to be at least $2n/3$ in order to guarantee 
that every $H$ on $n=3k$
vertices with $\theta(H)=4$ is a subgraph of $G.$ 

In~\cite{KY3} Kostochka and Yu proved the following:

\begin{theorem}
Each $n$-vertex graph $G$ such that $$\theta(G)\le \frac{2n}{3}-1$$ packs with every $n$-vertex graph $H$ such that
$\theta(H)\le 4.$
\end{theorem}

Let $G_1=\overline{G},$ then $\delta(G_1)\ge 2n/3.$ Hence, Conjecture~\ref{sejtes} is a corollary of the above 
theorem for the case $\theta(H)=4.$ Let us briefly mention another way of proving Conjecture~\ref{sejtes} for this case. 

For $\theta(H)\le 3$ we could use the existence of Hamilton cycles in the host graph in order to embed $H.$ 
For the case $\theta(H)=4$ one can first find the square of a Hamilton path, and then find $H$ in this special
structure. Denote $P_n^2$ the  square of a Hamilton path on $n$ vertices.
The following fairly simple result holds, we omit the proof. 

\begin{proposition}\label{propi}
Suppose that $H$ is a simple graph on $n$ vertices with $\theta(H)\le 4$. Then $H\subseteq P_n^2$.
\end{proposition}

So if we find the square of a Hamiltonian path in a graph $G$ with order $n$ then we can find an arbitrary subgraph with Ore-degree at most 4 in $G$. The following theorem was proved by Fan and Kierstead in~\cite{FK}. 

\begin{theorem}
If $G$ is a simple graph on $n$ vertices such that  $\delta(G)\ge\frac{2n-1}{3}$ then 
$P_n^2\subseteq G.$
\end{theorem}

Hence we have found an alternative proof of  Conjecture~\ref{sejtes} for the case $\theta(H)\le 4.$ 

An even stronger theorem was proved by Chau~\cite{Chau} (for very large graphs, the proof uses the Regularity lemma). He proved that the square of a Hamilton path packs with a graph $G$ on $n$ vertices whenever $\theta(G)\le2n/3 -1.$ 
Let us mention that recently DeBiasio, Faizullah and Khan~\cite{DFK} proved a similar result for packing the square of a Hamilton cycle without using the Regularity lemma. 

Finally, let us briefly mention the case when $H$ contains only vertex disjoint $K_{1,r}$s for some fixed $r\in \mathbb{N}.$ 
It turns out that regardless of the value of $r$ (as soon as it is a constant), it is sufficient if the minimum degree of $G$ is slightly larger than $n/2.$ Note that the Ore-degree of $H$ is 
$r+1,$ so it can be arbitrarily large.\footnote{A more sophisticated case is analyzed in~\cite{Va} for embedding a collection of stars, each star having at most $o(n/\log n)$ leaves.}

\begin{proposition}
\label{karom}
Let $r\in \mathbb{N}$ and $\epsilon>0.$ Then there exists an $n_0=n_0(r, \epsilon)$ such that if $G$ is a graph on $n$ vertices with 
$\delta(G)\ge (1/2+\epsilon)n,$ $n\ge n_0(r, \epsilon),$ and $(r+1)|n$ then $G$ contains a $K_{1, r}$-factor.
\end{proposition}

\begin{proof}
Randomly divide the vertex set of $G$ into the disjoint sets $X$ and $Y$ such that $r|X|=|Y|.$ We will embed the $K_{1,r}$-factor
into $G[X, Y].$ By the Azuma-Hoeffding bound (see Lemma~\ref{AzumaHoeffding} later) we get that if $n$ is sufficiently large 
then with positive probability 
for every $v\in X$ we have $\deg(v, Y)\ge (1/2+\epsilon/2)|Y|$ and for every $u\in Y$ we have that $\deg(u, X)\ge (1/2+\epsilon/2)|X|.$ 

Let us construct a bipartite graph $G'$ with vertex classes $X'$ and $Y.$ We obtain $X'$ 
by blowing up the set $X$ as follows: for every $v\in X$ we will have $r$ copies $v_1, \ldots, v_r\in X'.$ If $vu\in E(G)$ then
we will have all the $v_iu$ ($1\le i\le r$) edges in $G'.$ There are no other edges in $G'.$ 

Clearly, $G'$ has a perfect matching if and only if $G[X, Y]$ has a $K_{1, r}$-factor. The existence of
a perfect matching in $G'$ follows easily by verifying the K\"onig-Hall conditions, hence we proved what was desired.   
\end{proof}

Let us remark that the above result can also be proved by a routine application of the Regularity Lemma -- Blow-up Lemma method,
however, the proof presented here works for much smaller values of $n.$

\section{A review of tools for the proof}\label{Review}

First we take a short review of the tools we use.
The proofs of Theorem~\ref{main} and Theorem~\ref{stab} use the Regularity Lemma of Szemer\'edi~\cite{SzRL}. While below we
provide a brief introduction to the subject, the reader may also want to consult with the survey paper by Koml\'os and Simonovits~\cite{KS}.

If $A$ and $B$ are disjoint subsets of $V(G)$ then we denote by $e(A,B)$ the number of edges with one endpoint in
$A$ and the other in $B$. The density between disjoint sets $X$ and $Y$ is defined as
$$d(X,Y)=\frac{e(X,Y)}{|X|\cdot|Y|}.$$

We need the following definition to state the Regularity Lemma.

\begin{definition}[Regularity condition] Let $\varepsilon>0$. A pair $(A,B)$ of disjoint vertex sets of $G$ is $\varepsilon$-regular 
if for every $X\subset A$ and $Y\subset B$ satisfying
$$|X|>\varepsilon |A|,\quad |Y|>\varepsilon |B|$$
we have
$$|d(X,Y)-d(A,B)|<\varepsilon.$$
\end{definition}

We will employ the fact that if $(A,B)$ is an $\varepsilon$-regular pair as above, and we place at most $\varepsilon|A|$ new vertices 
into $A$, the resulting pair will remain $\varepsilon'$-regular, with $\varepsilon' \le \sqrt{\varepsilon}.$

An important property of regular pairs is the following:
\begin{lemma}\label{regfokok}
Let $(A,B)$ be an $\varepsilon$-regular pair with density $d$. Then for any $Y\subset B$ with $|Y|>\varepsilon|A|$ we have
$$|\{x\in A : \deg(x,Y)\le (d-\varepsilon)|Y|\}|\le \varepsilon|A|.$$
\end{lemma}

We will use the following form of the Regularity Lemma~\cite{SzRL,KS}:
\begin{lemma}[Degree form]\label{reglemma}
For every $\varepsilon>0$ there is an $M=M(\varepsilon)$ such that if $G=(V,E)$ is any graph and $d\in [0,1]$
is any real number, then there is a partition of the vertex set $V$ into $\ell+1$ clusters $V_0,V_1,\ldots,V_\ell$, 
and there is a subgraph $G'$ of $G$ with the following properties:
\begin{itemize}
\item $\ell\le M$,
\item $|V_0|<\varepsilon|V|$,
\item all clusters $V_i, i\ge 1$, are of the same size $m$ (and therefore $m\le\lfloor\frac{|V|}{\ell}\rfloor<\varepsilon|V|$),
\item $\deg_{G'}(v)>\deg_{G}(v)-(d+\varepsilon)|V|$ for all $v\in V$,
\item $V_i$ is an independent set in $G'$ for all $i\ge 1$,
\item all pairs $(V_i,V_j), 1\le i<j\le\ell$ are $\varepsilon$-regular, each with density either 0 or at least
$d$ in $G'$.
\end{itemize}
\end{lemma}

Often we call $V_0$ the exceptional cluster. In the rest of the paper we assume that $0<\varepsilon\ll d\ll 1$.

\begin{definition}[Reduced graph] Apply Lemma~\ref{reglemma} to the graph $G=(V,E)$ with parameters $\varepsilon$ and $d$ 
and denote the clusters of the resulting partition by $V_0,V_1,\ldots,V_\ell$, $V_0$ being the exceptional cluster. We construct 
a new graph $G_r$, the reduced graph of $G'$ in the following way: 
The non-exceptional clusters of $G'$ are the vertices of the reduced graph, hence $|V(G_r)|=\ell$. We connect two vertices 
of $G_r$ by an edge if the corresponding two clusters form an $\varepsilon$-regular pair with density at least $d$.
\end{definition}

The following corollary is immediate:
\begin{corollary}\label{mindeg}
Let $G=(V,E)$ be a graph of order $n$ and $\delta(G)\ge cn$ for some $c>0$, and let $G_r$ be the reduced graph of
$G'$ after applying Lemma \ref{reglemma} with parameters $\varepsilon$ and $d$. Then $\delta(G_r)\ge(c-2\varepsilon-d)\ell$.
\end{corollary}

\begin{definition}[Super-Regularity condition]
Given a graph $G$ and two disjoint subsets $A$ and $B$ of its vertices, the pair $(A,B)$
is $(\varepsilon,\delta)$-super-regular if it is $\varepsilon$-regular and furthermore
$$\deg(a)\ge\delta|B| \textrm{ for all } a\in A$$
and
$$\deg(b)\ge\delta|A| \textrm{ for all } b\in B.$$
\end{definition}

 Using Lemma~\ref{regfokok} it is easy to show that every regular pair contains an  ``almost spanning" super-regular pair, we leave the proof for the reader.

\begin{lemma}\label{szupreg}
Assume that $(A, B)$ is an $\varepsilon$-regular pair. %and let $d$ be a number such that $0\le \varepsilon \ll d \le d(A, B).$ 
Then there exists $A'\subset A$ and $B'\subset B$ such that 
$|A'|\ge (1-\varepsilon)|A|,$ $|B'|\ge (1-\varepsilon)|B|,$ and the $(A', B')$ pair is 
$(\varepsilon (1+\varepsilon), d(A, B)-2\varepsilon)$-super-regular.  
\end{lemma}

The Blow-up Lemma of Koml\'os, S\'ark\"ozy and Szemer\'edi~\cite{KSSz,KSSz2} asserts that dense super-regular pairs 
behave like complete bipartite graphs with respect 
to containing bounded degree subgraphs.  

\begin{theorem}[Blow-up Lemma]\label{blow-up}
Given a graph $R$ of order $r$ and positive parameters $\delta,\Delta$, there exists 
$\varepsilon=\varepsilon(\delta,\Delta,r)$ such that the following holds: 
Let $n_i$ for $i=1, \ldots, r$ be arbitrary positive integers and let us replace the vertices
$v_1,v_2,\ldots,v_r$ of $R$ with pairwise disjoint sets $V_1,V_2,\ldots,V_r$ of sizes
$n_1,n_2,\ldots,n_r$ (blowing up). We construct two graphs on the same vertex set 
$V=\bigcup_i V_i$. The first graph $F$ is obtained by replacing each edge $\{v_i,v_j\}$
of $R$ with the complete bipartite graph between $V_i$ and $V_j$. A sparser graph $G$
is constructed by replacing each edge $\{v_i,v_j\}$ arbitrarily with an 
$(\varepsilon,\delta)$-super-regular pair between $V_i$ and $V_j$. If a graph $H$ with  
$\Delta(H)\le\Delta$ is embeddable into $F$ then it is already embeddable into $G$.
\end{theorem}

The Blow-up Lemma is among the most important tools in many graph embedding algorithms. However, for proving our main result
we also need a somewhat different, technically more involved version, introduced in~\cite{Cs2}. In order to state this lemma, 
we need some preparations.

Let $G'$ and $H$ be graphs\footnote{In fact $G'$ is the graph we obtain from $G$ by applying the Regularity Lemma and doing 
further preparations, 
while $H$ is the graph we want to embed into $G.$} 
of order $n.$ Assume that $V(G')=V_0\cup V_1\cup\ldots\cup V_\ell$ and $V(H)=L_0\cup L_1\cup\ldots\cup L_\ell$ are partitions such that 
there is a bijective mapping $\psi_0: L_0\to V_0,$ furthermore, $|V_i|=|L_i|=m$ for every $1\le i\le\ell.$ 

\begin{definition}[$(d,\varepsilon)$-goodness]
Let $x\in L_i$; a vertex $v\in V_i$ is called $(d,\varepsilon)$-good for $x$ if $y\in N(x)\cap L_j$ implies 
$\deg_{G'}(v,V_j)\ge(d-\varepsilon)m$ for every $1\le j\le\ell.$
\end{definition}

Assume that $D=\Delta(H)\ge 1,$ and let $\hat{I}\subset V(H)$ be a maximal set the elements of which are of distance at least 5 
from each other. 
Using the above assumptions and notations, the modified version of the Blow-up Lemma is as follows.

\begin{theorem}[Modified Blow-up Lemma \cite{Cs2}]\label{modblow-up}
For every positive integer $D,$ $K_1,$ $K_2,$ $K_3$ and every positive constant $c$ there exist $n_0$ such that if 
$\varepsilon, \varepsilon', \delta, d$ are positive constants with 
$$0<\varepsilon\ll\varepsilon'\ll\delta\ll d\ll 1/D, 1/K_1, 1/K_2, 1/K_3, c$$ then the following holds.

Suppose that $G'$ and $H$ are graphs of order $n$ with partitions defined as above such that $n\ge n_0.$ %and $|\hat{I}|\gg d n.$ 
Suppose further 
that for every $1\le i<j\le\ell$ the pair $(V_i,V_j)$ is 
$\varepsilon$-regular with density 0 or $d.$  Furthermore, suppose that the following conditions hold. 

\begin{itemize}
\item[{$\mathrm{\mathbf{C1}}$:}] $|L_0|=|V_0|\le K_1dn$;
\item[{$\mathrm{\mathbf{C2}}$:}] $L_0\subset \hat{I}$;
\item[{$\mathrm{\mathbf{C3}}$:}] $L_i$ is independent for every $1\le i\le\ell;$ 
\item[{$\mathrm{\mathbf{C4}}$:}] $|N_H(L_0)\cap L_i|\le K_2dm$ for every $1\le i\le\ell$;
\item[{$\mathrm{\mathbf{C5}}$:}] for every $1\le i\le\ell$ there is $B_i\subset \hat{I}\cap L_i$ with $|B_i|=\delta m,$ 
such that for $B=\bigcup_iB_i$ and every $1\le i,j\le\ell$
$$|\,|N_H(B)\cap L_i|-|N_H(B)\cap L_j|\,|\le\varepsilon m;$$ 
\item[{$\mathrm{\mathbf{C6}}$:}] if $(x,y)\in E(H)$ and $x\in L_i$, $y\in L_j$ for $1\le i, j\le \ell,$ then  $(V_i,V_j)$ is 
an $\varepsilon$-regular pair with density $d;$
\item[{$\mathrm{\mathbf{C7}}$:}] if $(x,y)\in E(H)$ and $x\in L_0$ then $y\in L_j$ ($j>0$) implies $\deg_{G'}(\psi_0(x),V_j)\ge c m;$
\item[{$\mathrm{\mathbf{C8}}$:}] for every $1\le i\le\ell$, given any $E_i\subset V_i$ such that $|E_i|\le\varepsilon'm$ there 
exists a set 
$F_i\subset (L_i\cap (\hat{I}-B))$ and a bijection
$\psi_i:E_i\to F_i$ such that for every $v\in E_i$, $v$ is $(d,\varepsilon)$-good for $\psi_i(v);$
\item[{$\mathrm{\mathbf{C9}}$:}] for $F=\bigcup_iF_i$ we have that
$$|N_H(F)\cap L_i|\le K_3\varepsilon'm.$$
\end{itemize}

\noindent Then $H$ could be embedded into $G'$ such that the image of $L_i$ is $V_i$ for every $1\le i\le \ell,$ and the image of each 
$x\in L_0$ is $\psi_0(x)\in V_0.$
\end{theorem}

Let us give some remarks on the lemma. First, we need this version since it does not demand super-regularity between cluster pairs, 
we have conditions
C6 and C7 instead. It enables us to directly work with the graph we obtain by applying the Regularity Lemma, even though the number 
of clusters depends on 
the regularity constant $\varepsilon.$ 

Condition C1 requires that the number of ``exceptional'' vertices is small. By condition C2 we make sure that embedding the vertices 
of $L_0$ can be done without 
affecting the neighborhood of other vertices of $L_0.$ It is clear that we need $C3$ since every cluster of $G'$ is an independent set. 
Condition C4 ensures 
that embedding $L_0$ will not significantly affect the embedding of other $L_i$ sets. By condition C5 we can have sufficiently large 
sets of buffer vertices (the vertices of $B$) 
such that the vast majority of them is embedded in the end. The role of C8 and C9 is to eliminate a possible objection during the 
embedding. Finally, C7 ensures that the neighbors 
of the $L_0$ vertices have 
sufficiently large space. 

The reader may notice that we have not defined
which vertices would belong to $E_i$ (conditions C8 and C9), we only made assumption on the size of these sets -- the $E_i$ sets are determined during 
the execution of the algorithm as follows.
The embedding of $H$ is done sequentially: first  we map the vertices of $N_H(L_0),$ then $N_H(B),$ the neighbors of the buffer vertices. 
These are small sets, hence, 
without any difficulty we can map them. At this point we look at $G$ and identify $\cup E_i,$ the exceptional or ``risky'' 
vertices\footnote{Roughly speaking, 
a vertex of $G$ is exceptional, if after the mapping of $N_H(L_0)$ and $N_H(B),$ it is not contained in the vacant neighborhood of many 
buffer vertices.}, and cover them right away 
with buffer vertices. The set of buffer vertices we use for this purpose is denoted by $F.$ The vast majority of the vertices of $H$ is 
mapped after taken care of these exceptional vertices.

\medskip

P.~Hajnal, S.~Herdade, A.~Jamshed and E.~Szemer\'edi in~\cite{HHJSz} proved the following stability version of the P\'osa-Seymour conjecture.
Before stating it, we need the notion of {\it $(\eta, k)$-extremal} graphs. Recall that we have already defined
$\eta$-extremality of a graph, which is in fact equivalent to $(\eta, 3)$-extremality, as one can see immediately from the 
definition below. 

\begin{definition}
Given some integer $k\ge 3$ and real $\eta>0$ we say that graph 
$G$ of order $n$ is 
{\em $(\eta, k)$-extremal} if there exists $A\subset V(G)$ with $|A|=\lfloor n/k\rfloor$ such that 
$e(G[A])\le \eta {|A| \choose 2}< \eta n^2/(2k^2).$ If 
such a subset $A$
does not exist, we say that $G$ is {\em $(\eta, k)$-non-extremal}.
\end{definition}

\begin{theorem}\label{Endre}
Let $k\ge 3$ be an integer and $\eta>0$ be a real number.
  There exists an integer $n_0(\eta,k)$, and positive real number $\gamma(\eta,k)$ such
  that any $(\eta, k)$-non extremal graph $G$, with
  $v(G)=n>n_0(\eta,k)$, and having minimum degree 
  $\delta(G)\geq (\frac{k-1}{k}-\gamma(\eta,k))n$, 
  contains a $(k-1)^{th}$ power of a Hamiltonian cycle.
\end{theorem}

The authors of~\cite{HHJSz} remark that Theorem~\ref{Endre} works when $\gamma=\eta^{1000}$ although a much smaller exponent would suffice. 
When applying their theorem, we will assume that $\eta$ and $\gamma$ are related according to this equation. 

We mostly consider $(\eta, 3)$-extremal and
$(\eta, 3)$-non-extremal host graphs, and, as we indicated above, we may refer to them as $\eta$-extremal or $\eta$-non-extremal graphs. 

Depending on the structure of $G$ and $H$ the embedding algorithm has several cases. Hence,
we will also need another notion of extremality, for the graph $H$ to be embedded.    

%\begin{definition}[$\eta$-extremality]
%We say that a graph $G$ of order $n$ is $\eta$-extremal, if there exists a set $A\subset V(G)$ with $|A|=\lfloor\frac{n}{3}\rfloor$ 
%such that $e(G[A])<\eta\cdot n^2.$
%\end{definition}

\begin{definition}[$\nu$-triangular extremality]
We say that a graph $H$ on $n$ vertices is $\nu$-\emph{triangular extreme} if it contains at least $(1-\nu)n/3$ vertex disjoint triangles, 
here $0<\varepsilon\ll d\ll \nu\ll 1.$
\end{definition}

A proper vertex coloring of a graph is \emph{equitable} if the sizes of its color classes differ by at most one. Kierstead and Kostochka proved the following.

\begin{theorem}[Kierstead, Kostochka, \cite{KK}]\label{equicol}
If $H$ is a graph having $\theta(H)\le 2k+1$ then it has an equitable $(k+1)$-coloring.
\end{theorem}

As our embedding method is in fact a randomized algorithm, we are going to use large deviation bounds, for example the well-known Chernoff inequality, see~\cite{AS}. 
We will also use another, less common inequality, a corollary of the Azuma inequality. We will refer to it as the Azuma-Hoeffding bound (sometimes it is 
called Hoeffding's bound~\cite{Chaz}):

\begin{lemma}\label{AzumaHoeffding}
Let us assume that we are given an urn with $r$ red and $b$ blue balls. Let $N = r + b.$ We
conduct the following experiment: randomly, uniformly draw $m$ balls (where $1\le m \le N$) without
replacement. Denote the number of chosen red balls by $X.$ Then for every $0\le q\le N$ we have 
$$\mathbb{P}\left(|X-\mathbb{E}X| \ge q \right)\le 2e^{-\frac{q^2}{2m}}.$$
\end{lemma}

It is easy to see, that $\mathbb{E}X = ma/N.$

\medskip

% Ezt nem tudom, jo lesz-e meg valamire
%\begin{theorem}[Fan, Kierstead, \cite{FK}]\label{hamilton}
%Every graph $G$ on $n$ vertices satisfying $\delta(G)\ge\frac{2n-1}{3}$ contains the square of $P_n$.
%\end{theorem}

\section{Proof of the main results when $H$ is not triangular extreme}\label{non-extrem}

The goal of this section is to prove Theorem~\ref{main} and Theorem~\ref{stab} for the case when $H$ is not triangular extreme. We will achieve this goal through several subsections. The triangular extreme case will be handled in Section~\ref{extrem}.

Here is a sketch of the main ideas of this section. In subsection~\ref{5.1} we apply the Regularity lemma for $G$ and obtain the reduced graph $G_r.$ We cover the vast majority of $G_r$ with disjoint $K_3$s. In subsection~\ref{5.2} we decompose $H,$ as a result we will obtain a large independent set $I$ such that the components of $H-I$ are paths of length at most 2. In subsection~\ref{homomorfizmus}, using random methods and matchings, we find a homomorphic mapping $f:V(H) \longrightarrow V(G_r)$ such that $f$ assigns almost the same number of vertices of $H$ to every clusters of $G_r.$ Finally, in subsection~\ref{5.4} after somewhat involved preparations based on $f$  we apply the Modified Blow-up lemma in order to finish the embedding.

\subsection{Preprocessing the host graph}\label{5.1}
Recall that in Theorem~\ref{stab} (the stability version), the minimum degree of the host graph is allowed to be slightly smaller than $2n/3.$
Our strategy for proving Theorem~\ref{main} is to first prove the stability version, and then consider only extremal host graphs
in order to finish the proof, since the stability version implies Theorem~\ref{main} whenever $G$ is non-extremal.

We will only use the Regularity Lemma for the stability version, hence, we will apply it for a host graph $G$ having minimum degree
$\delta(G)\ge (2/3-\gamma)n$ for some\footnote{The value of $\gamma$ is determined by the parameters of the Regularity Lemma, as we will see soon.} $\gamma>0,$  however, throughout the section we do not assume that $G$ is non-extremal\footnote{Hence the results of this section hold for an extremal host graph as well.}. 

\begin{itemize}
\item[\textbf{Step 1}] Given $0<\varepsilon \ll d\ll 1,$ let us set $\gamma=d-\varepsilon,$ here $\gamma$ is the parameter of the stability version and $\varepsilon$ and $d$ are the parameters of the Degree Form of the Regularity Lemma.
Applying the Degree Form of the Regularity Lemma for $G$ with the above parameters we obtain $\ell'+1$ clusters, $V_0,V_1,\ldots,V_{\ell'}$ where $V_0$ is the exceptional cluster. Next we construct the reduced graph $G_r$. Recall that $V(G_r)=\{V_1,\ldots,V_{\ell'}\}$ and $V_iV_j\in E(G_r)$ if $(V_i,V_j)$ form an $\varepsilon$-regular pair in $G$ with density at least $d$.
Observe that by Corollary~\ref{mindeg} the minimum degree
of the reduced graph $G_r$ is $\delta(G_r)\ge (2/3-2d)\ell'.$

\item[\textbf{Step 2}] Let us now add $6d\ell'$ fictive vertices to the vertex set of $G_r.$ Connect each of the fictive
vertices to each of the original vertices of the reduced graph. The resulting new graph has $(1+6d)\ell'$ vertices and its
minimum degree is at least $(2/3-2d)\ell' + 6d\ell'=(2/3+4d)\ell'={2 \over 3}(1+6d)\ell',$ hence it satisfies the degree condition of the
Corr\'adi-Hajnal theorem, and therefore it contains a triangle factor.

\item[\textbf{Step 3}] Next we delete all those triangles from this factor that contain any of
the fictive vertices. The non-fictive vertices of
deleted triangles will be put into the exceptional cluster $V_0,$ this will increase its size by at most $12d\ell'm\le 12dn$ vertices of $G,$ so
we have $|V_0|\le 13dn.$ The minimum degree of $G_r$ may also decrease by at most $12d\ell'.$
Then the reduced graph on the remaining vertices (that are in fact non-exceptional clusters) has a triangle factor, which we denote by
$\mathcal{T}.$ For simpler notation we will still denote this modified reduced graph by
$G_r,$ and the clusters of it will be denoted by $V_1, \ldots, V_{\ell},$ where $\ell\ge (1-12d)\ell'.$

\item[\textbf{Step 4}] Recall, that every vertex of the reduced graph is a cluster on $m$ vertices. Using Lemma~\ref{szupreg} repeatedly, by deleting $4\varepsilon m$ vertices from each
cluster in the triangles of $\mathcal{T}$
we can achieve that every cluster-edge in every triangle is a $(2\varepsilon, d-4\varepsilon)$-super-regular pair. For simpler notation we will denote the new cluster sizes by $m.$ Observe, that
when doing so we increased the size of the exceptional cluster
by at most $4\varepsilon n$ vertices, hence, we have that $|V_0|\le 14dn.$ The vertex set of the reduced graph (the non-exceptional clusters) and its triangle factor $\mathcal{T}$ have not changed.
\end{itemize}

Summarizing, we have the following.
\begin{lemma}\label{hszogfaktor}
After preprocessing $G,$ we obtain a reduced graph $G_r$ with vertices $V_1, \ldots, V_{\ell}$ (these are non-exceptional clusters in $G$) so that $|V_i|=m,$ and every edge in $G_r$ represents a $2\varepsilon$-regular pair. The minimum degree is
$\delta(G_r)\ge (2/3-14d)\ell.$
The reduced graph has a triangle factor $\mathcal T$ such that every edge of $\mathcal T$ represents a $(2\varepsilon, d-4\varepsilon)$-super-regular pair. Furthermore, the exceptional cluster $V_0$ has at
most $14 dn$ vertices.
\end{lemma}

\subsection{Structural properties of $H$}\label{5.2}
Observe that $\Delta(H)\le 4$. In fact the only vertices that can have 4 neighbors are centers of $K_{1,4}$s. The presentation of the proof will be
somewhat simpler later if we assume that the subgraph $H_1 \subset H$ that we obtain by deleting every $K_{1,4}$ from $H$ is {\it saturated,}
that is, we add edges to $H_1$ one-by-one until we obtain a subgraph with Ore-degree being equal to 5, but adding any new edge 
would make its Ore-degree larger (of course, we do not add parallel or loop edges). 

It is clear that we cannot add any edge that connects a vertex
of a $K_{1,4}$ to any other vertex without increasing the Ore-degree to 6. Similarly, if $T_1$ and $T_2$ are arbitrary vertex disjoint triangles in $H,$ then we 
cannot add any edge that connects $T_1$ and $T_2.$ 
 
Let $D_i=\{x\in V(H) : \deg_H(x)= i\}$ for $i=0, \ldots, 4.$ It is easy to see that at most two vertices
can have degree at most 1 in the saturated part of $H,$ therefore, we have that $|V(H_1)\cap (D_0\cup D_1)|\le 2.$
%Note that $|D_0|, |D_1|\le 1$ since $H$ is saturated.

%We may assume that $H$ is saturated, i.e. one cannot add edges to $H$ such that $\theta(H)\le 5$ remain. We note that $|D_0|\le 1$.

An important ingredient for proving Theorem~\ref{main} is the following decomposition of $H:$

\begin{lemma}\label{strH}
Let $H$ be a graph of order $n$ with $\theta(H)=5$. Then there exists $I\subset V(H)$ such that the following conditions hold:
\begin{itemize}
\item[{(1)}] $I$ is an independent dominating set in $H$ with $|I|\ge n/3$.
\item[{(2)}] $\deg_H(x)\le 2$ for each $x\in I$.
\item[{(3)}] The connected components of $H-I$ are paths with length at most 2.
\item[{(4)}]  If $x\in I$ with $N_H(x)=\{y_1,y_2\}$ then either $y_1y_2\in E(H)$ or $y_1$ and $y_2$ belong to different components of $H-I$.
\end{itemize}
\end{lemma}

\begin{proof}
Observe that if $x\in D_4$ then $x$ is the center vertex of a star in $H.$ Moreover, $D_3\cup D_4$ is 
an independent vertex set by the condition on the Ore degree.
Let
\begin{eqnarray*}
\mathcal{P}&=&\{xz_1\ldots z_r y  \textrm{ is a path of } H : r\ge 0, x\ne y, x,y\in D_1\cup D_3, z_1,\ldots,z_r\in D_2\}\\
&&\cup\ \{xz_1\ldots z_r x  \textrm{ is a cycle of } H : r\ge 2, x\in D_3, z_1,\ldots,z_r\in D_2\}.
\end{eqnarray*}

Therefore $H-(D_3\cup D_4)$ consists of disjoint cycles and paths (note that a path with length 0 is a vertex). We will consider the class of paths in $H-(D_3\cup D_4)$ and connect their endpoints with their neighbors in $D_3$.
Notice the $r=0$ case in the above definition. It e.g.~takes care of $K_{1,3}$s, as $xy_1, xy_2, xy_3$ are all paths that belong to 
$\mathcal{P}$ (here $x$ is the center vertex
of some $K_{1,3},$ the $y_i$s are the leaves). 

\begin{claim}\label{P-fok1}
If $x\in D_3$ and $xy\in E(H),$ then $y$ belongs to the vertex set of some $P\in \mathcal{P}$ such that $xy$ is an edge of $P.$
\end{claim}

\begin{proof} (of the Claim)
If $y\in D_1$ then by the definition of $\mathcal P$ we have $P=xy.$ Assume that $y\in D_2$ and $xy$ does not belong to any $P \in \mathcal{P}.$ Then we 
can build a path or cycle $P=xyz_1\ldots z_t$ greedily
such that $yz_1, z_iz_{i+1}\in E(H)$ for all $1\le i\le t-1,$ and $z_t\in D_1\cup D_3$
(if $z_t=x,$ then we found a cycle). Observe that there does not exist any $P'\in \mathcal{P}$ with $y\in P'$ such that $x\not\in P',$ 
since both neighbors of $y$ must belong to $P'.$ Hence, $y$ is
available when we build $P.$
\end{proof}

Let us consider the following auxiliary bipartite graph $B$: the color classes of $B$ are $D_3$ and $\mathcal{P}$, and $xP\in E(B)$ if 
and only if $x\in D_3$, $P\in \mathcal{P}$ and 
$x\in P.$ 

\begin{claim}\label{P-fok2}
If $x\in D_3$ then $2\le \deg_{B}(x)\le 3$,  moreover, if $P\in \mathcal{P}$ then $\deg_{B}(P)\le 2$.
\end{claim}

\begin{proof} (of the Claim)
First, observe that no $x\in D_3$ can have three neighbors in some $P\in \mathcal{P}:$ otherwise among the three neighbors one could 
find a vertex $y$ that itself has two neighbors in $P,$
implying that $y\in D_3,$ contradicting with the definition of $\mathcal{P}.$ Since by Claim~\ref{P-fok1} every neighbor of $x$ belong 
to some $P\in \mathcal{P},$ we get the 
lower bound for $\deg_B(x).$ 
It is clear, that $\deg_B(x)\le 3.$ Finally, $\deg_B(P)\le 2,$ since at most two vertices of $P$ do not belong to $D_2$ by definition. 
\end{proof}

Counting the number of edges between $X$ and 
$\mathcal{P}$ in $B$ shows that for every $X\subset D_3$ we have $|X|\ge |N_{B}(X)|$, just  Hence, by the K\"onig-Hall marriage theorem 
$B$ has a $D_3$-saturating matching $M.$
Let $$I'=\{y : (x,P)\in M, y\in P, x\in N_H(y)\}.$$

\noindent The following are immediate:
\begin{itemize}

\item[($i$)] $|I'|=|D_3|,$

\item[($ii$)] every $x\in D_3$ has a neighbor in $I',$

\item[($iii$)] every $y\in I'$ has a neighbor in $D_3,$

\item[($iv$)] $I' \subseteq D_1\cup D_2$ is an independent set.

\end{itemize}

If $x\not\in I'$ then $\deg_{H-I'}(x)=4$ or $\deg_{H-I'}(x)\le 2,$ moreover, if  $\deg_{H-I'}(x)=4$ then $x$ is the center vertex of
some $K_{1,4}.$ Therefore the components of $H-I'$ are paths, cycles and stars, implying that
the chromatic number of $H-I'$ is at most 3. 

Let $I$ be an independent set with maximal size such that $I'\subseteq I.$ 
Then we have $$|I-I'|\ge \frac{1}{3}(|V(H-I')|-2|I'|)$$ since $I-I'$ is a maximal independent set which does not contain any vertex 
of $I'\cup N_H(I'),$ $|N_H(I')|\le 2|I'|$ by ($iv$) above and we also use that $\chi(H-I'-N_H(I'))\le 3.$ 
 Therefore $$|I|\ge\frac{1}{3}(|V(H-I')|-2|I'|)+|I'|=
\frac{n}{3}.$$ Moreover if $x\in V(H-I)$ then $x$ has a neighbour in $I$ by ($ii$) and the maximality of $I$ so condition (1) holds.

Clearly, $D_3\cap I=\emptyset$ by ($ii$). If $x\in D_4$ then $x\not\in I,$ furthermore, $N_H(x)\subseteq I-I'$ by the construction 
of $I'$ and the maximality of $I$. 
Therefore condition (2) holds, implying that the components of $H-I$ are paths and cycles.

Suppose that $C\subset H-I$ is a cycle. By the condition on the Ore-degree there exists an $x\in V(C)$ with $\deg_H(x)=2,$ therefore 
$I\cup x$ is an independent set with larger size which contradicts to the maximality of $I$.

Similarly, suppose that $P\subset H-I$ is a path $x_0\ldots x_k$ with $k\ge 3$. Then by the condition on the Ore-degree there exists an $0<i<k$ 
with $\deg_H(x_i)=2$ therefore  $I\cup x_i$ is an independent set with larger size which contradicts to the maximality of $I$, therefore 
condition (3) holds.

Suppose that $x\in I$ with $N_H(x)=\{y_1,y_2\}$. Suppose further that $y_1y_2\not\in E(H)$ and $y_1$ and $y_2$ belong to the same component 
of $H-I.$ Using condition (3) this component have to be a path $y_1x'y_2.$ If $x'\in D_2$ then $I\cup x'$ is an independent set that contains $I'$ and larger than $I,$ this contradicts to the maximality of $I$. 
If $\deg_H(x')>2$ then $y_1,y_2\in D_2$ since $\theta(H)\le 5$ and $xy_i, x'y_i\in E(H)$ for $i=1,2.$ This implies that $x\not\in I'$ using ($iii$).
This allows us to find a new independent set $(I-\{x\})\cup\{y_1,y_2\}$ which is clearly larger than $I$ and also contains $I'$, contradicting to the maximality of $I.$ Hence, condition (4) holds.
\end{proof}

\medskip

Let 
$$I_1=\{x\in I\cap D_2 : y_1,y_2\in N_H(x), P(y_1)=P(y_2)\}$$ and 
$$I_2=\{x\in I\cap D_2 : y_1,y_2\in N_H(x), P(y_1)\neq P(y_2)\}.$$

Observe, that $I_1$ contains exactly one vertex of each triangle of $H$ by Lemma~\ref{strH}, moreover, every vertex of $I_1$ belongs to some triangle. Clearly, $I_1\cup I_2=I\cap D_2$.

%and $$I_2=\{x\in I\cap D_2 : y_1,y_2\in N_H(x), P(y_1)\neq P(y_2)\}.$$
%
%Finally, let 
%$I_1^1=I-I_1-I_2-D_0.$ Clearly, $I_1^1=I\cap D_1,$ since $I\subset (D_0\cup D_1\cup D_2).$ 

In order to apply the Modified Blow-up Lemma, we need a subset of $I-I_1-D_0$ that contains vertices which are at distance at least 5 from each other. Denote $\widehat{I}$ the maximum sized subset of $I-I_1-D_0$ having this property.
%from $x$ by the condition on the Ore degree, the independence of $I$ and because $\deg_H(x)\le 2$.

\begin{claim}\label{fgetlen2}
If $H$ is not $\nu$-triangular extreme, then
$|\widehat{I}|\ge \frac{\nu n}{40}.$
\end{claim}

\begin{proof}
For every $x\in I$ there are at most 12 vertices in $I$ at distance at most 4 from $x,$
therefore, using a simple greedy algorithm one can see that
 $$|\widehat{I}|\ge\frac{|I-I_1|}{13}-1\ge \frac{\nu n/3}{13}-1\ge \frac{\nu n}{40}$$
by Lemma~\ref{strH}, using also that $|I_1|\le (1-\nu)n/3,$ since $H$ is not triangular extreme. We subtracted 1 since $|D_0|\le 1.$ 
\end{proof}

Since $\widehat{I}\subset I-I_1,$ depending on the structure of $H,$ it is possible that $\widehat{I} \cap I_2$ is small. This 
happens only when $H$ mainly contains $K_{1,4}$s (recall, that $H_1$ is saturated). 

The set $\widehat{I}$ plays an important role in the embedding procedure of $H,$ when applying the Modified Blow-up Lemma. Since we will use it for three different purposes, 
it is useful to divide it randomly into
three disjoint parts, $\widehat{I}_1, \widehat{I}_2$ and $\widehat{I}_3,$ each having either $\lfloor|\widehat{I}|/3\rfloor$ or $\lceil|\widehat{I}|/3\rceil$ vertices, so  $\widehat{I}=\widehat{I}_1\cup \widehat{I}_2 \cup \widehat{I}_3.$ 

If $\widehat{I}$ were too small, our embedding algorithm would not work. In order to have that $\widehat{I}$ is sufficiently large, we set $\nu=\sqrt[4]{\gamma} \ (=\sqrt[4]{d-\varepsilon}).$  

\subsection{Homomorphism from $H$ to $G_r$}\label{homomorfizmus}

We say that a homomorphism $f: V(H)\longrightarrow V(G_r)$ is \emph{balanced} if $$\left|\,|f^{-1}(V_i)|-|f^{-1}(V_j)|\,\right|\le\varepsilon^2 m$$ for every $V_i,V_j\in V(G_r),$
here $\varepsilon, \ell$ are the parameters of the Regularity Lemma.  
Such a balanced homomorphism from $H$ to $G_r$ plays a key role in our embedding procedure. We find it in two steps. First we determine $f: V(H)-I \longrightarrow V(G_r),$
and then $f:I\longrightarrow V(G_r).$ 

We will use the Modified Blow-up lemma (Lemma~\ref{modblow-up}) for embedding $H$ into $G.$ The $V_i$ clusters will be ``almost'' the partition
sets of the Modified Blow-up Lemma. We also need to partition the vertex set of $H$ into $L_0, L_1, \ldots, L_{\ell}.$ In this  section we are going to discuss in detail how to find this partition. When $f$ is at hand, the sets $f^{-1}(V_i)$ for $1\le i\le \ell$ will be ``first approximations" of the $L_i$ sets, but for obtaining the final $L_i$ sets we may need to redistribute a small proportion of the vertices\footnote{For example, $f(V(H))\cap V_0$ is empty at the moment, while $V_0$ and therefore $L_0$ are non-empty sets in general.}. 

It is notationally convenient to introduce another function, $h,$ which maps $V(H)$ to the set $\{0, 1, \ldots, \ell\}$ such that $h(x)=i$ if $x\in L_i,$ that is, $h(x)$ is the {\it index} of the set containing $x$ in the partition of $V(H).$  Up to a certain point in the distribution of $V(H)$ the two functions, $f$ and $h,$ are in the following relation: $f(x)=V_i$ if and only if $h(x)=i.$

\subsubsection{Assigning the paths in $H-I$ and the vertices of $I-I_2 -\widehat{I}$}

Roughly speaking, in order to find $f$ restricted to $V(H)-I$ we will randomly assign the vertices 
of $H-I$ to the vertices of $G_r,$ so that components of $H-I$ are assigned to clusters of some  triangle of the triangle factor $\mathcal{T}.$

More precisely, let $P$ be a path in $H-I$ containing vertices $x_i$ for $1\le i\le k=|P|+1$, here $|P|\le 2$ by Lemma~\ref{strH}. 
Pick uniformly at random a triangle $T$ from $\mathcal{T}.$ Denote $V_{s_1}, V_{s_2}, V_{s_3}$ the vertices of $T.$ 
Then pick uniformly at random a permutation $\pi$ on $\{1, 2, 3\},$ and let $f(x_i)=V_{s_{\pi(i)}}$ for every $1\le i\le k.$ 
Set $f(x_i)=V_{s_{\pi(i)}}$ for $1\le i\le k$ and let $h(x_i)=s_{\pi(i)}.$

%For a vertex $y\in H-I$ we let $P(y)$ denote the path in $H-I$ that contains $y.$ 
%Recall that
%$I_2\subset (I\cap D_2)$ denotes the set of those vertices that have their neighbors in two different components of $H-I$ and 
%$I_1=(I\cap D_2)-I_2,$ moreover, every vertex of $I_1$ belongs to a triangle.

%, hence, the edge of that triangle  component of $H-I.$ 

Suppose now that $x\in I-I_2-\widehat{I}.$ Then we have already assigned each vertex of $N_H(x)$  to the same 
triangle $T\in {\mathcal T}$ randomly, if $x\not\in D_0.$ Since $\deg_H(x)\le 2$ there will be at least one vertex of $T$ that has no 
neighbor of $x$ assigned to it. If there are two such vertices, pick one of them randomly. In both cases denote $V_s$ the chosen vertex, 
and let $f(x)=V_s$ and $h(x)=s.$ If $x\in D_0$ then choose $T$ and then $V_s$ from $T$ randomly.

Later we will need the following.

\begin{lemma}\label{varhato}
 Let $i\in \{1, 2, \ldots, \ell\}$ be arbitrary. Assume that $x\in V(H)-I_2-\widehat{I}.$ Then $$\Pr(h(x)=i)=\frac{1}{\ell}.$$ Assume further that $x', x''\in V(H)-I$ such that
$x'$ and $x''$ belong to different components of $H-I.$ Then $$\Pr(h(x')=h(x'')=i)=\frac{1}{\ell^2},$$ and $$\Pr(h(x')=h(x''))=\frac{1}{\ell}.$$
Similarly, if $y, y'\in V(H)-I_2-\widehat{I}$ 
such that the distance of $y$ and $y'$ is at least five, then $$\Pr(h(y)=h(y')=i)=\frac{1}{\ell^2},$$ and $$\Pr(h(y)=h(y'))=\sum_i \Pr(h(y)=h(y')=i) =\frac{1}{\ell}.$$ 
\end{lemma}

\begin{proof}
It is clear that the vertices of $H-I$ are distributed uniformly at random, moreover, if $x', x''$ belong to different components then $h(x')$ and $h(x'')$ are independent. From this the second and third equations follow immediately. For the first equation
we need to consider vertices of $I_1, D_0$ and leaves of $K_{1,4}$s as well. Since such vertices have at most one neighbor which is uniformly distributed among the vertices
of $G_r$ (as we have just seen), these must also be uniformly distributed, using the assigning method of this section. Finally, if $y$ and $y'$ are at distance at least five then $h(y)$ and $h(y')$ define independent, uniformly distributed random variables, so we obtain what was desired.  
\end{proof}

\subsubsection{Assigning the vertices of $I_2\cup \widehat{I}$}\label{propmatching}

We have not taken care of all the vertices of $D_1,$ since $\widehat{I}\cap D_1$ may not be empty.
Let $\widehat{I}^s$ denote the subset of 
$\widehat{I}$ that consists of leaves belonging to $K_{1,4}$s of $H.$ Note that every $K_{1,4}$ contributes to 
$\widehat{I}^s$ with exactly one leaf.

Set $I_2'=\{x\in I_2:|f(N(x))|=1\}$. 
In order to simplify the discussion regarding the assignment of the vertices of $I_2'\cup \widehat{I}^s,$ we are going to 
introduce {\it fictive} neighbors\footnote{We are not going to embed the fictive vertices, these are used only for assigning
their neighbors in $I_2'\cup \widehat{I}^s.$} for them. 

This goes as follows. Let $F$ denote the set of fictive neighbors of the vertices of $I_2'\cup \widehat{I}^s.$
There is a bijection $g: I_2'\cup \widehat{I}^s \longrightarrow {F},$ every $x\in I_2'\cup \widehat{I}^s$ has exactly one fictive neighbor
$g(x)\in {F},$ and every $y\in {F}$ is the fictive neighbor of exactly one $x=g^{-1}(y)\in I_2'\cup \widehat{I}^s.$ Denote this expanded graph by $H^+.$  We also let $I_F=  I_2'\cup \widehat{I}^s.$

We are going to distribute the vertices of ${F}$ among the $L_i$ sets randomly: for every $y\in {F}$ we randomly, uniformly, independently from the other choices pick an index $j\in \{1, \ldots, h(N(g^{-1}(y)))-1, h(N(g^{-1}(y)))+1, \ldots, \ell\},$ assign $y$ to $V_j,$ and also let $h(y)=j.$  
With introducing fictive neighbors and distributing them randomly we have in fact achieved that every vertex of 
$I_F$ behaves as it were from $I_2-I_2'.$

After these preparations we are ready to discuss the distribution of the vertices of $I_2\cup I_F$ to vertices of the reduced graph in a balanced way.
This is a significantly harder task than the assignment of $I-I_2-I_F,$ which was considered before. For that we will use proportional (or many-to-one or star) and strong proportional matchings 
(this method was perhaps used first 
in~\cite{CsSSz}, and also played important role in~\cite{Cs2}) in appropriately defined auxiliary graphs. 

\begin{definition}[proportional matching \cite{CsSSz}]
Let $\F(R, S)$ be a bipartite graph with $|S|=q|R|$ for some positive integer $q.$ We say that $\M\subset E(\F(R,S))$ is a
\emph{proportional matching} if every $v\in R$ is adjacent to exactly $q$ vertices in $S$ and every $u\in S$ is adjacent to exactly one $v\in R$ in $\M$.
\end{definition} 

The following claim is a simple consequence of the K\"onig-Hall marriage theorem.

\begin{claim}[Proportional K\"onig-Hall]\label{propKH}
We have a proportional matching in the graph $\F(R, S)$ defined above if and only if $|N_{\F}(A)|\ge q|A|$ for every $A\subset R.$
\end{claim}

\begin{proof}
The proof is very similar to a part of the proof of Proposition~\ref{karom}. We construct a new bipartite graph $\F',$ one vertex class of it is $S,$ the other is $R',$ the latter is the {\it blown-up} $R.$ That is, for every $v\in R$ we will have $q$ copies, $v_1, \ldots, v_q\in R'.$ We have an edge $v_iu\in E(\F')$ for $v_i\in R'$ and $u\in S$ if and only if $vu\in E(\F),$ where $v_i$ is
a copy of $v.$ It is clear that a perfect matching in $\F'$ is a proportional matching in $\F$ and vice versa, and the K\"onig-Hall conditions for $\F'$ translate to the proportional K\"onig-Hall conditions of the claim. 
\end{proof}

We are going to need an auxiliary bipartite graph.

\begin{definition}[$\Lambda_1$-graph]
The $\Lambda_1=\Lambda_1(G_r)$ graph is a bipartite graph having vertex classes $V(G_r)$  
and $\mathcal{S}=\binom{V(G_r)}{2},$ that is, $\mathcal S$ is the set of all unordered 2-element subsets of $V(G_r).$  We have an edge $WS \in E(\Lambda_1)$ for $W\in V(G_r)$ and $S\in \mathcal{S}$ if and only if 
$W$ is adjacent to both vertices of $S$ in $G_r.$ 
\end{definition}

The lemma below is a special case of Lemma 20 in~\cite{Cs2}.
\begin{lemma}\label{bipart}
There is a proportional matching $\M_1$ in $\Lambda_1.$
\end{lemma}

\begin{proof}
We are going to check the proportional K\"{o}nig-Hall conditions given in Claim~\ref{propKH}.
%, i.e. for every subset $A$ of $V(G_r)$ its neighborhood in $S$ should satisfy 
%$|N_{\Lambda_1}(A)|\ge |A|\cdot|S|/|V(G_r)|$ 
%in order to prove that the claimed matching exists. 

\begin{itemize}

\item Let $W\in V(G_r)$ be an arbitrary vertex. By the minimum degree
condition on $G_r$ we have that $W$ is adjacent to at least 
$$\left(\frac{2}{3}(1-14d)\right)^2|\mathcal{S}|-o(|\mathcal{S}|)>0.4 |\mathcal{S}|$$ pairs in $\mathcal S$ (count the number of pairs in $N_{G_r}(W)$).

\item Next we take an arbitrary set $A\subset V(G_r)$ with $|A|=0.4 \ell.$ Then any $U\in V(G_r)$ will have a neighbor in $A.$ Say, $W\in A$ such that $UW\in E(G_r).$ 
Then for  every $U'\in N_{G_r}(W)$ the $(U,U')$ pair is adjacent to $W,$ hence, 
$$|N_{\Lambda_1}(A)| \ge \frac{2}{3}(1-14d)|\mathcal{S}|.$$

\item Assume that $A\subset V(G_r)$ with $|A|=\frac{2}{3}(1-14d)\ell.$ Consider an auxiliary graph $F$ with vertex classes $A$ and $V(G_r),$ where for a $W\in A$ 
and a $U\in V(G_r)$ we have $WU\in E(F)$ if and only if $WU\in E(G_r).$ We get that $e(F)\ge |A|\cdot \delta(G_r)\ge (\frac{2}{3}(1-14d))^2\ell^2.$ Using this lower bound for $e(F)$ simple calculation shows
that more than 30\% of the vertices of $V(G_r)$ must have more than $(1/3+100d)\ell$ neighbors in $A.$ 

It is easy to see that if a vertex $W\in V(G_r)$ has more than $(1/3+100d)\ell$ neighbors in $A,$ then for every $U\in V(G_r)$ the $(W,U)\in \mathcal{S}$ pair is adjacent to some 
vertex of $A.$
This implies that $|N_{\Lambda_1}(A)|\ge 0.7 |\mathcal{S}|.$ 

\item Finally, let $A\subset V(G_r)$ with $|A|=0.7 \ell.$ Then every $(W, W')\in \mathcal{S}$ is a neighbor of some $U \in A,$ therefore, $N_{\Lambda_1}(A)=\mathcal{S}.$

\end{itemize} 
\end{proof}

We need another kind of matching about which we demand that it ``distributes" the vertices at least slightly more evenly.

\begin{definition}[strong proportional matching]
%Consider the bipartite 
%graph $\Lambda_1=\Lambda_1(G_r)$ with vertex classes $V(G_r)$ and $S=\binom{V(G_r)}{2}.$ Set $q=|S|/\ell$ (recall, that $|V(G_r)|=\ell$). 
Let $\mu$ be a real such that $0<\mu< 1.$
We say that $\Lambda_1$ allows a strong proportional matching with respect to 
$\mu$ if there is a proportional matching in the following bipartite graph $\Lambda_2=\Lambda_2(\Lambda_1)$. Its color classes are $V(G_r)$ and $\mathcal{S}'$, where we obtain 
$\mathcal{S}'$ from $\mathcal{S}$ in the following way. For every element
$U\in \mathcal{S}$ we add $\frac{\ell}{\mu}$ copies $U_1,\ldots,U_{\frac{\ell}{\mu}}$ to $\mathcal{S}',$ hence, 
$|\mathcal{S}'|=|\mathcal{S}|\cdot \ell/\mu.$ If $N_{\Lambda_1}(U)=\{W_1,\ldots, W_t\}$ then we will have the following edges: 
$(U_i,W_i)$ for $1\le i\le t$, and $(U_j,W_i)$ for $1\le i\le t$ and $t< j\le \frac{\ell}{\mu}$. In other words, the first $t$ copies of $U$ have degree 1, while the others have the same degree, $t$.
We will refer to $\Lambda_2$ as the $\Lambda_2$-graph of $G_r.$
\end{definition}

%The proof of the following claim is left for the reader: 

\begin{claim}\label{atteres}
Let $\Lambda_1$ and $\Lambda_2$ be graphs as above. 
If $A\subset V(G_r)$ then $$\frac{|N_{\Lambda_1}(A)|(1-\mu)}{|\mathcal{S}|}\le \frac{|N_{\Lambda_2}(A)|}{|\mathcal{S}'|}.$$ 
\end{claim}
\begin{proof}
If $U\in N_{\Lambda_1}(A)$ then $U_{\ell+1},\ldots, U_{\frac{\ell}{\mu}}\in N_{\Lambda_2}(A),$
therefore $$|N_{\Lambda_2}(A)|\ge \ell\left(\frac{1}{\mu}-1\right)|N_{\Lambda_1}(A)|.$$
Recall that $|\mathcal{S}'|=|\mathcal{S}|\cdot \ell/\mu,$ so the claim follows.  
\end{proof}

Using this fact we can prove the existence of a strong proportional matching relatively easily, using the existence of a proportional matching in $\Lambda_1.$
We set the parameter $\mu$: let $\mu=\nu \ (=\sqrt[4]{d-\varepsilon}),$ then $0<\varepsilon, d \ll \mu \ll 1.$
We have the following.

\begin{lemma}\label{strong-bipart}
 The $\Lambda_2$-graph of $G_r$ has a proportional matching, hence, $\Lambda_1$ allows a strong proportional matching $\M_2.$
\end{lemma}

\begin{proof}
Note that $q=|\mathcal{S}'|/\ell=|\mathcal{S}|/\mu.$ 
Let $A\subset V(G_r)$ be arbitrary. We are going to show that 
$$\frac{|N_{\Lambda_2}(A)|}{|\mathcal{S}'|}\ge \frac{|A|}{\ell},$$ which implies the existence of the desired strong proportional matching in $\Lambda_2.$ 

\begin{itemize}

\item Suppose first that $|A|\le (1-\mu)\ell.$ One can see from the proof of Lemma~\ref{bipart} that 
$$\frac{|N_{\Lambda_1}(A)|}{|\mathcal{S}|}\ge (1+2\mu)\frac{|A|}{\ell}$$ since $\mu$ is small. 
Then we can use Claim~\ref{atteres} and get that
$$\frac{|N_{\Lambda_2}(A)|}{|\mathcal{S}'|}\ge (1-\mu) (1+2\mu)\frac{|A|}{\ell}>\frac{|A|}{\ell}.$$ 

\item Assume now that $|A|>(1-\mu)\ell.$ Observe first, that every $S\in \mathcal{S}$ that has degree larger than
1 is adjacent to some $V_i\in A.$ Let $W\in V(G_r)-A.$ Then $\mathcal{S}'$ contains at most $|\mathcal{S}|$ elements that are adjacent to 
$W$ in $\Lambda_2$ and have degree 1.
Hence, overall $\mathcal{S}'$ has at most $(|V(G_r)-A|)\cdot |\mathcal{S}|$ elements that are not adjacent to some vertex of $A$ in $\Lambda_2.$ Hence,
$$\frac{|N_{\Lambda_2}(A)|}{|\mathcal{S}'|}\ge 1-\mu\left(1-\frac{|A|}{\ell}\right)\ge 1-\mu\ge \frac{|A|}{\ell},$$ since 
$$|\mathcal{S}|\cdot \ell/\mu -|\mathcal{S}|(\ell-|A|)=|\mathcal{S}'|(1-\mu+\mu |A|/\ell).$$
\end{itemize}
\end{proof}

\medskip

%Up to a certain point in the distribution of $V(H)$ the two functions, $f(x)=V_{h(x)},$ .

Let us first discuss the assignment of the vertices of $(I_2\cup I_F)-\widehat{I}_1.$ Consider the $\Lambda_1$-graph of $G_r$ and apply Lemma~\ref{bipart}.
Recall that $\M_1$ denotes the proportional matching of $\Lambda_1$ and $\M_2$ denotes the strong proportional matching. 
Assume that $x\in (I_2\cup I_F)-\widehat{I}_1$ and $h(N_{H^+}(x))=\{i, j\}.$ 
Denote $V_k\in V(G_r)$ the cluster to which the pair $(V_{i}, V_{j})$ is matched in $\M_1.$ Then we let $f(x)=V_k$ and $h(x)=k.$ 

It is more complicated to assign vertices of $\widehat{I}_1$ to the $L_i$ sets. Let $i, j$ be fixed indices such that 
$1\le i<j\le \ell.$ Define the following subset of $\widehat{I}_1:$
$$\widehat{I}_1(i, j)=\left\{x\in \widehat{I}_1: h(N_{H^+}(x))=\{i, j\}\right\}.$$
Notice that $\widehat{I}_1(i, j)$ contains those vertices of $\widehat{I}_1$ that have their neighbors in $L_i$ and $L_j.$
Next take a {\it random} equipartition of $\widehat{I}_1(i, j)$ into the disjoint sets $S_1, \ldots, S_{\ell/\mu}.$ That is,
$$\widehat{I}_1(i, j)=S_1\cup \ldots \cup S_{\ell/\mu},$$ and $||S_t|-|S_r||\le 1$ for every $1\le t, r \le \ell/\mu,$ where the $S_t$
sets are random subsets. For example, one can find this random partition in the following way: take uniformly at 
random a permutation $\pi$ on $\widehat{I}_1(i, j).$ Next take consecutive segments of lengths 
$\lfloor \mu |\widehat{I}_1(i, j)|/\ell\rfloor$ or $\lceil \mu |\widehat{I}_1(i, j)|/\ell\rceil,$
starting from the first element of $\widehat{I}_1(i, j)$ according to $\pi.$ The elements of the $t$-th segment will be the set $S_t.$
 
If the $t$-th copy of $(V_i, V_j)$ is matched to $V_k$ in $\mathcal{M}_2$ then we let $f(x)=V_k$ and $h(x)=k$ for 
every $x\in S_t.$
We repeat the above for every $i, j$ pair. 

\smallskip

Since adjacent vertices of $H^+$ are assigned to adjacent vertices of $G_r,$ with this we have found a homomorphism $f: V(H)\longrightarrow V(G_r).$ 
What is left is to prove that $f$ is balanced.

%\medskip

\begin{lemma}\label{balansz}
Let $A\subseteq V(H)$ and $k \in \{1, \ldots, \ell\}$ be arbitrary. Then $$\Pr\left(\left||A\cap L_k|-\frac{|A|}{\ell}\right|\ge n^{3/4}\right)\le \frac{1}{\sqrt{n}}.$$ 
%for every $1\le i\le \ell$, moreover $$\Pr\left(\left||X\cap L_i|-\frac{|X|}{\ell}\right|< n^{3/4} \textrm{ for every} 1\le i\le \ell \right)\ge 1- \frac{\ell}{\sqrt{n}}.$$ 
\end{lemma}

\begin{proof} Let $\{x_1,\ldots, x_a\}$ denote the vertices of $A$ (hence $|A|=a$). We define {\it  indicator random variables} $X_1, \ldots, X_a$ such that for every 
$i\in \{1, \ldots, a\}$ we have $X_i=1$ if and only if
$h(x_i)=k.$ Set $X=\sum_i X_i.$ It is easy to see that $X=|A\cap L_k|.$ Our first goal is to show the following.

\begin{claim}\label{varhato2} 
The expected number of vertices of $A$ assigned to $L_k$ is $\mathbb{E}X=\frac{|A|}{\ell}.$
\end{claim}

\begin{proof} (of the Claim)
We show that $\mathbb{E}X_i=1/\ell$ for every $i.$ Whenever $x_i\in A-I_2-\widehat{I}$ this is an immediate consequence of Lemma~\ref{varhato}. 

Assume now that $x_i\in A\cap (I_2-I_2'-\widehat{I}),$ and denote the neighbors of $x_i$ 
by $y$ and $y'.$ By Lemma~\ref{varhato} we have that the $(h(y), h(y'))$ pair is uniformly distributed in the set of all possible two-element subsets of $\{1, \ldots, \ell\}.$ For finding $h(x_i)$ we use the proportional matching ${\mathcal M}_1.$ Since
every vertex of $G_r$ has the same degree, $(\ell-1)/2$ in ${\mathcal M}_1$ and every pair has degree 1, we get that $\mathbb{E}X_i=1/\ell$ in this case, too. 
The case when $x_i \in A\cap (I_2'-\widehat{I})$ is very similar. Since the non-fictive neighbors of $x_i$ are assigned to a randomly, uniformly chosen vertex of $G_r$ by Lemma~\ref{varhato}, the fictive neighbor must also be uniformly distributed. Hence, the argument we used for vertices from $I_2-I_2'-\widehat{I}$ works here as well. 

Assume that $x_i\in A\cap (\widehat{I}-\widehat{I}_1).$ For assigning such an $x_i$ we use ${\mathcal M}_1,$ and since
the neighbors (fictive or non-fictive) of $x_i$ are uniformly distributed in $\binom{V(G_r)}{2},$ similarly to previous cases we conclude that $h(x_i)$ is uniformly distributed.  

Finally, assume that $x_i\in A\cap \widehat{I}_1.$ As before, we know that its neighbors are uniformly distributed in $\binom{V(G_r)}{2}.$ Recall the definition of the $\widehat{I}_1(i, j)$ sets. When distributing the vertices of $\widehat{I}_1$ using the strong proportional matching
we first divide the $\widehat{I}_1(i, j)$ sets randomly into $\ell/\mu$ subsets, this procedure is independent from the random distribution of the vertices in $V(H)-I.$ Since the probability that 
$x_i \in \widehat{I}_1(i, j)$ is exactly $1/\binom{\ell}{2},$ we get that $h(x_i)=k$ with probability $1/\ell$ when using ${\mathcal M}_2$ for finding the function $h.$ Hence, $h(x_i)$ is uniformly distributed in this case, too.

Since $\mathbb{E}X_i =\frac{1}{\ell}$ for every $i,$ using linearity of expectation finishes the proof of the claim.
\end{proof}

In order to finish the proof of the lemma   
we need a simple claim, a direct consequence of Chebyshev's inequality.

\begin{claim}\label{Csebisev}
Let $k, s\in \mathbb{N}$ be fixed such that $s\ge k.$ Assume that $Z_1, \ldots, Z_s$ are indicator random variables such that 
$\Pr(Z_i=1)=p$ for every $i.$ Assume further that every $Z_i$ is independent from at least $s-k$ other indicator variables.
Set $Z=\sum_i Z_i.$
Then $$\Pr(|Z-\mathbb{E}[Z]|\ge \lambda \sqrt{(k+1)s p})\le \frac{1}{\lambda^2}$$ for every $\lambda>0.$ 
\end{claim}

\begin{proof} (of the Claim)
We estimate $\mathrm{Var}[Z]$ as follows: $$\mathrm{Var}[Z]=\sum_i \mathrm{Var}[Z_i]+\sum_{i\neq j}(\mathbb{E}[Z_iZ_j]-\mathbb{E}[Z_i]\mathbb{E}[Z_j]).$$
Let's consider the terms of the first sum: $$\mathrm{Var}[Z_i]=\mathbb{E}[Z_i^2]-\mathbb{E}^2[Z_i]=p-p^2\le p$$ holds for every $i.$
Next we consider the covariances. Whenever $Z_i$ and $Z_j$ are independent, we have that $\mathbb{E}[Z_{i}Z_{j}]-\mathbb{E}[Z_{i}]\mathbb{E}[Z_{j}]=0.$
When they are not independent, we will use the following bound: 
$$\mathbb{E}[Z_{i}Z_{j}]-\mathbb{E}[Z_{i}]\mathbb{E}[Z_{j}]\le \mathbb{E}[Z_i]-\mathbb{E}[Z_{i}]\mathbb{E}[Z_{j}]=p-p^2\le p$$ which holds for every $i$ and $j.$ Putting these together we obtain
that $\mathrm{Var}[Z]\le (k+1)s p.$ Applying Chebyshev's inequality finishes the proof. 
\end{proof}

Now let $x_i$ and $x_j$ be two vertices of $A-\widehat{I}_1.$ Observe that if $X_i$ and $X_j$ are not independent random variables then there exists a component $P$ of $H-I$ such that  $x_i, x_j\in V(P)\cup N_H(V(P)).$ 
By Lemma~\ref{strH} we have that if $x ,x' \in V(P)\cup N_H(V(P))$ then $\dist(x, x')\le 4,$ hence, for each $x\in V(H)$ there are less than $3^4$ vertices with distance at most 4.
Set $\lambda=2\sqrt[4]{n}$ then Claim~\ref{Csebisev} implies that  
$$\Pr\left(\left||(A-\widehat{I}_1)\cap L_k|-\frac{|A-\widehat{I}_1|}{\ell}\right|\ge\frac{n^{3/4}}{2}\right)\le \frac{1}{2\sqrt{n}}.$$

Finally, we discuss the vertices in $A\cap \widehat{I}_1,$ which is a slightly more complicated case. Consider a pair of indices  
$(s, t)$ such that $V_sV_k, V_tV_k\in E(G_r)$ (therefore, a copy of the $(V_s, V_t)$ pair is matched to $V_k$ in the strong proportional matching $\M_2$). The random variables that correspond to $x_i\in A\cap \widehat{I}_1$ are independent, since
$\widehat{I}_1$ contains only such vertices that are at distance at least 5 from each other. We get that 
$$\Pr\left(\left||A\cap \widehat{I}_1(s, t)|-\frac{|A\cap \widehat{I}_1|}{{\ell \choose 2}}\right|>\sqrt{n}\log n \right)<\frac{1}{n},$$ here we used the Chernoff bound (see e.g.~\cite{AS}) for 
the case $|A\cap \widehat{I}_1|\ge \sqrt{n}\log n,$ and the trivial bound $\sqrt{n}\log n$ if $A\cap \widehat{I}_1$ is too small. This implies that after distributing the vertices of 
$\widehat{I}_1$ using $\M_2,$ the following will hold with probability at least $1/(2\sqrt{n})$: the number of vertices
of $A\cap \widehat{I}_1$ assigned to $L_k$ differs from its expectation by at most ${n}^{3/4}/2.$ This proves the lemma.
\end{proof}

We also get the following:
 
\begin{corollary}
For every $W\in V(G_r)$ we have $|\,|f^{-1}(W)-F|-n/\ell|=O(n^{3/4})$ with probability at least $1-\ell/n^{1/4}.$
\end{corollary}

%---------------------- BLOW-UP LEMMAS RESZ ---------------------------------------

\subsection{Finishing the embedding}\label{5.4}

In this section we will apply the Modified Blow-up lemma (Theorem~\ref{modblow-up} in this paper) in order to embed a non-triangular extreme graph $H.$
Recall, that $H$ is non-$\nu$-triangular extreme, if it contains less than $(1-\nu)n/3$ vertex disjoint triangles, where $d\ll \nu \ll 1.$
The Modified Blow-up lemma has several conditions, we will go through each of them to verify that they 
are satisfied. 

Whenever it is possible, we will refer to lemmas, claims of~\cite{Cs2}, since the Modified Blow-up lemma was introduced in that paper,
moreover, proportional and strong proportional matchings were used for distributing vertices among clusters of the host graph, exactly like in the present paper.  

An ``almost final" partition of $V(G),$ required by the Modified Blow-up lemma, is naturally given by applying the Regularity Lemma and then preprocessing of $G_r$ in Lemma \ref{hszogfaktor}. The exceptional set is $V_0,$ the $V_i$ sets for $1\le i\le \ell$ are the non-exceptional clusters. It is possible that we need to move around a small proportion of the vertices of $G$ among the 
clusters.

A {\it provisional} partition $V(H)=L_1\cup \ldots \cup L_{\ell}$ is provided by the balanced homomorphism 
$f: V(H)\longrightarrow V(G_r)$ of Section~\ref{homomorfizmus}. Recall, that 
$L_i=f^{-1}(V_i).$
It is clearly temporary, since $L_0$ is empty at this point.

\subsubsection{Bad vertices in $G$ -- preparation for conditions C8 and C9}

By definiton of the reduced graph, if $V_iV_j\in E(G_r)$ then the $(V_i, V_j)$ pair is $\varepsilon$-regular. However, 
$V_i$  may have up to 
$\varepsilon m$ vertices that has only a small number of neighbors in $V_j,$ or even no neighbor at all (of course, similar holds for $V_j$). 
In order to avoid problems that may be caused by this, we will discard a few vertices from the non-exceptional clusters, and place them into $V_0.$ 

The procedure we use for determining the vertices to be placed to $V_0$ is based on the following notion.
Given the proportional matching $\M_1$ provided by Lemma~\ref{bipart}, for every $1\le i\le \ell$ let $S_i$ denote the set of pairs matched to the non-exceptional cluster $V_i$ in $\M_1.$ We say that a vertex\footnote{An {\it ordinary} vertex of $G,$ not a cluster.} $v\in V(G)-V_0$ has {\it $\alpha$-small degree to a pair} $S\in \binom{V(G_r)}{2},$ if $v$ has less than $(d-\alpha)m$ neighbors in at least one
of the clusters of $S.$ A vertex $v\in V_i$ is called {\it $\alpha$-bad,} if $v$ has $\alpha$-small degree to at least $|S_i|/2$ pairs in $S_i.$
The lemma below is from~\cite{Cs2} (can be found as Lemma 4.7), we omit the proof.

\begin{lemma}\label{BadVertices}
By removing \emph{exactly} $4\varepsilon m$ appropriately chosen vertices from every non-exceptional cluster of $G'$ we can achieve that no $6\varepsilon$-bad vertices will remain in them.
\end{lemma}

\begin{corollary} \label{bad} After performing the above procedure the edges of $G_r$ will represent $2\varepsilon$-regular pairs with density at least $d-4\varepsilon,$ moreover, $$3\varepsilon n< 4\varepsilon \ell m\le |V_0|\le 4\varepsilon m\ell+14d n < 15d n.$$ 
%Moreover, $$\sum_{i\ge 1}\left(|L_i|-|V_i|\right)=|V_0|>3\varepsilon n.$$ 
\end{corollary}

\subsubsection{Forming $L_0$ and satisfying conditions C1, C2, C3, C4, C6 and C7} 

By Lemma~\ref{balansz} we have that $||\widehat{I}_j\cap L_i|-|\widehat{I}_j|/\ell|=o(n)$ for every $1 \le i \le \ell$ and $j=1, 2, 3$ with high probability. The sets $\widehat{I}_1,$ $\widehat{I}_2$ and $\widehat{I}_3$ will be used for different purposes. In this section we show how to fill up $L_0$ with vertices of $\widehat{I}_1.$
 
First, take a random subset $R_i\subset \widehat{I}_1\cap f^{-1}(V_i)$ of size\footnote{Recall that we chose $\varepsilon, d$ and $\nu$ such that $\varepsilon\ll d\ll \nu.$ One can see that 
$|\widehat{I}_1\cap f^{-1}(V_i)|> |L_i|-|V_i|$ is always satisfied by Claim~\ref{fgetlen2} and Lemma~\ref{BadVertices}. } $|L_i|-|V_i|$ for every 
$1\le i \le \ell,$ and let $L_0=R_1\cup \ldots \cup R_{\ell}.$ Clearly, we have achieved that $|L_0|=|V_0|, |L_1|=|V_1|, \ldots, |L_{\ell}|=|V_{\ell}|$ since $v(H)=v(G).$ 

At this point C1 is satisfied by Remark~\ref{bad}, C2 and C3 hold by definition and C6 is satisfied since the components of $G-I$ are distributed randomly among $L_i$ sets so that if $xy\in E(H)$ for $x, y\in V(G)-I,$ then  $V_{h(x)}V_{h(y)}\in E(G_r),$ and the vertices of $I$ are assigned to $L_i$ sets using proportional and strong proportional matchings. 

The following lemma proves that C4 holds with high probability.

\begin{lemma}\label{L_0szomszed1}
For every $1\le i \le \ell$ we have $|N_H(L_0) \cap L_i|\le 3|L_0|/\ell=3|V_0|/\ell<50dm$ with high probability.
\end{lemma}

\begin{proof}
Fix an arbitrary index $ i\in \{1, \ldots, \ell\}.$ We have $$|R_j|=|L_j|-|V_j|\le\ell^2 n^{3/4}+\frac{|V_0|}{\ell}$$ for each $1\le j\le\ell$ by Lemma~\ref{balansz}. Therefore
 if $x\in \widehat{I}_1\cap L_j$ then $$\Pr(x\in R_j)=\frac{|R_j|}{|\widehat{I}_1\cap L_j|}\le \frac{\frac{|V_0|}{\ell}+\ell^2n^{3/4}}{\frac{|\widehat{I}_1|}{\ell}-\ell^2n^{3/4}}\le\frac{|V_0|}{|\widehat{I}_1|}(1+\varepsilon),$$
the first inequality follows from Lemma~\ref{balansz}, the second inequality holds because $|V_0|\ge 3\varepsilon n$ by 
 Lemma~\ref{BadVertices} and Corollary~\ref{bad} and $|\widehat{I}_1|\ge\nu n/40$ by Claim~\ref{fgetlen2}.
So if $x\in\widehat{I}_1$ then $p_x=\Pr(x\in L_0)\le \frac{|V_0|}{|\widehat{I}_1|}(1+\varepsilon).$

For each $x\in\widehat{I}_1$ we define a random variable: $Z_x=|N_H(x)\cap L_i|$ if $x$ is chosen for $L_0$ and 0 otherwise. Note that $0\le Z_x\le 2$. Clearly $Z=\sum_{x\in\widehat{I}_1} Z_x=|N_H(L_0)\cap L_i|$. Let us estimate the expectation of $Z.$

\begin{eqnarray*}
\mathbb{E}(Z)&=&\sum_{x\in\widehat{I}_1}\mathbb{E}(Z_x) = \mathop{\sum_{
x\in\widehat{I}_1}}_{|N_H(x)\cap L_i|=1} p_x + \mathop{\sum_{x\in\widehat{I}_1}}_{|N_H(x)\cap L_i|=2} 2p_x\\
&\le& |N_H(L_0)\cap L_i|\cdot\frac{|V_0|}{|\widehat{I}_1|}(1+\varepsilon)\\
&\le& \frac{|N_H(\widehat{I}_1)|}{\ell}(1+\varepsilon)\cdot\frac{|V_0|}{|\widehat{I}_1|}(1+\varepsilon)\le \frac{2|V_0|}{\ell}(1+\varepsilon)^2,
\end{eqnarray*}
where the last two inequalities come from Lemma~\ref{balansz} and the fact that $|N_H(\widehat{I}_1)|\le 2|\widehat{I}_1|$.
Observe that if $x, y\in \widehat{I}_1$ then $\mathbb{E}(Z_xZ_y)\le\mathbb{E}(Z_x)\mathbb{E}(Z_y).$ 
It is also clear that $Var(Z_x)\le 4p_x(1-p_x)$ for every $x\in \widehat{I}_1$ using the definition of variance.
Therefore 
$$Var(Z)= \sum_{x\in \widehat{I}_1}Var(Z_x)+\sum_{x, y\in \widehat{I}_1, \ x\neq y} \left(\mathbb{E}(Z_xZ_y)-\mathbb{E}(Z_x)\mathbb{E}(Z_y)\right)$$ 
$$\le \sum_{x\in \widehat{I}_1}Var(Z_x)\le\sum_{x\in \widehat{I}_1}4p_x(1-p_x)<n.$$
Applying the Chebyshev inequality we get that
$$\Pr(|\,Z-\mathbb{E}(Z)|\ge\lambda \sqrt{n})\le \frac{1}{\lambda^2}.$$ Substituting 
$\lambda=\sqrt[4]{n}$ proves the lemma.\hfill
\end{proof}

Since we may not have condition C7 at this point, further work is required.
Recall the {\it index function} $h.$
Let us fix an arbitrary bijective mapping $\psi_0: L_0 \longrightarrow V_0.$ For all $x\in L_0$ we have to check whether condition C7 of the Blow-up Lemma holds,
that is, we need that $\deg_{G'}(v, V_{h(y_1)}), \deg_{G'}(v, V_{h(y_2)}) \ge c m,$ where $y_1, y_2$ are the two neighbors of  $x$ and $v=\psi_0(x).$ 

If this condition does not hold for some $x$ then a {\it switching} will be performed. Switching goes as follows. First, uniformly at random we pick a cluster
$V_i$ from the common neighborhood of $V_{h(y_1)}$ and $V_{h(y_2)}.$ Note that this common neighborhood contains more than $\ell/4$ clusters (in fact almost $\ell/3$
clusters even in the worst case).

Then locate a vertex $x'\in L_i\cap \widehat{I}_1$ such that $\deg_{G'}(v, V_{h(y'_1)}), \deg_{G'}(v, V_{h(y'_2)})\ge c m$, where $x'y'_1, x'y'_2 \in E(H).$ 
This is done randomly: we pick uniformly at random a pair of vertices  from the neighborhood of $V_i$ among those into which $v$ has at least $c m$ neighbors. Denote the set
of these vertices of $G_r$ by $R.$ It is easy to see
that $\deg(v, V_j)\ge c m$ for at least $5\ell/9$ vertices $V_j$ of $G_r,$  since $c$ is small (we can choose $c=100 \nu,$ say), and out of this many vertices only 
slightly more than $\ell/3$ are ruled out that are not neighbors of $V_i.$ Hence, we have more than $\ell/5$ vertices in $R$ from which we can choose one {\it available} pair randomly. 
 
As soon as we have the pair of vertices, say $V_j$ and $V_k,$ the strong proportional matching we used to distribute $\widehat{I}_1$ allows us to find an $x'$ 
with the property required above. Here is a brief calculation: $|\widehat{I}_1|\ge \nu n/150,$ hence, $|\widehat{I}_1(j, k)|\ge \nu n/(75 \ell^2)$ (we divided by ${\ell \choose 2}$),
one copy of a pair in $R$ therefore ``sends" at least $\nu \mu n/(75\ell^3)$ vertices of $\widehat{I}_1,$ since the number of copies is $\ell/\mu$ in the definition of the $\Lambda_2$ graph.

The number of copies of available pairs in $R$ is at least ${\ell/5 \choose 2}\approx \ell^2/50.$ Hence, the number of candidates for switching is larger than $\nu \mu n/(4000\ell)$ for every $x\in L_0$ 
for a given $V_i.$ Since we chose $V_i$ from the common neighborhood of $V_{h(y_1)}$ and $V_{h(y_2)},$ overall we have at least about $\nu \mu n/12000$ candidates for the vertices of $L_0.$ Hence if we choose $\nu, \mu$ and $d$ such that $d\ll \nu\mu \ll 1,$ then we can perform the switching. Since we have $\nu=\mu \approx \sqrt[4]{d},$ we do not get stuck when switching.    

Once we have $x'$, we will switch the roles of $x$ and $x',$
that is, we let $L_i=L_i+x-x',$ $L_0=L_0-x+x'$ and $\psi_0(x') =v.$
It is clear that we have  condition C7 for $x',$ and that we reassigned $x$ to $L_i$ so that we did not violate other conditions of the Blow-up Lemma.

The above calculation shows that we do not get stuck during switching even if every vertex of $L_0$ has to be switched. The randomness in the procedure also guarantees that condition
C4 will hold at the end. The lemma below is proved (in two separate statements, Lemma 4.8 and Lemma 4.9) in~\cite{Cs2}:

\begin{lemma}
Using the above switching procedure, for every $x\in L_0$ one can find an $x'\in \widehat{I}_1-L_0$ so that at the end of switching for every $1\le k\le \ell$ with high probability we 
have that
$$|L_k \cap N_H(L_0)|\le Kdm$$ where $K\le 15000.$
\end{lemma}

It is easy to see that conditions C1, C2, C3 and C6 are kept during the switching procedure.

\subsubsection{Condition C5}

In order to satisfy condition C5 we find a set $B \subset \widehat{I}_2.$ Recall that in $\widehat{I}_2$ vertices have their neighbors in two different $L_i$-sets (one of these neighbors could be a fictive one). In the beginning let $B=\emptyset.$

Let $(V_i, V_j),$ $i\neq j$ be a pair of clusters and consider the set $$S_{i, j}= \widehat{I}_2\cap N_{H^+}(L_i)\cap N_{H^+}(L_j).$$ The vertices of $S_{i, j}$ were assigned to clusters
using the proportional matching $\Lambda_1.$
Pick $\delta m \ell/\binom{\ell}{2}$ vertices from $S_{i, j}$ arbitrarily 
and place them into $B.$ Repeat the above step for every possible cluster pair. Finally, let $B_i=B\cap L_i$ for all $1\le i\le \ell.$
It is easy to see that the following lemma holds, we omit the proof:

\begin{lemma}
We have $|B_i|=\delta m$ and $|N_H(B)\cap L_i|=2\delta m.$
\end{lemma}

Hence if we determine $B$ and the $B_i$ sets using the above procedure we can satisfy condition C5.

\subsubsection{Conditions C8 and C9}

We briefly explained the role of conditions C8 and C9 right after the Modified Blow-up lemma. As above, we will refer to a 
lemma\footnote{ Only straightforward modifications were made in the notation, that is, we use $\varepsilon'$ instead of $\varepsilon''$ and $\widehat{I}_3$ instead of $B_2'.$}, Lemma 4.11 in~\cite{Cs2} which proves
that conditions C8 and C9 are satisfied: 

\begin{lemma}
Given arbitrary sets $E_i\subset V_i$ such that $|E_i|\le \varepsilon' m,$ we can find the sets $F_i\subset L_i\cap \widehat{I}_3$ and bijective mappings $\psi_i: F_i\longrightarrow E_i$ for every $1\le i\le \ell$ such that the following hold: 

\noindent (1) if $xy\in E(H)$ with $x=\psi_i^{-1}(v)$ and $y\in L_j$ then $\deg(v, V_j)\ge (d-6\varepsilon)m,$

\noindent (2) for $F=\cup F_i$ we have $|N_H(F)\cap L_i|\le 6\varepsilon' m.$ 
\end{lemma}

\medskip

With this we have finished proving an important special case of Theorem~\ref{main} and Theorem~\ref{stab}:

\begin{proposition}\label{harmadik}
There exist positive numbers  $\gamma, \nu, n_0$ such that if $n > n_0$ and $\theta(H)\le 5$ for a 
graph $H$ of order $n$ that is not $\nu$-triangular extreme, and $\delta(G)\ge (2/3-\gamma)n$ for any   
graph $G$ of order $n$ then $H\subset G.$ 
\end{proposition}

%\begin{remark}
 Let us emphasize that according to the above result, if $H$ has relatively many vertices that do not belong to triangles, then
$H\subset G$ even if $\delta(G)$ is somewhat smaller than $2n/3.$ It does not even matter whether $G$ is extremal or not.
Basically, the reason of this is that if $H$ is not extremal then we can distribute a non-negligible number of vertices by the help of
the strong proportional matching and can use switching when necessary, so we can fill up $L_0.$ If $H$ were extremal, we could still embed the vast majority of it into
$G,$ but perhaps some vertices would be left uncovered in $V_0.$ 
%\end{remark}

%---------------------- EXTREM ESET -----------------------------------------------

\section{Embedding in the triangular extreme case}\label{extrem}

In order to embed a triangular extreme $H$ into $G$ we will consider two cases. In the first one we assume that $G$ is not $\eta$-extremal, while in the second case $G$ is an $\eta$-extremal graph. The second case is considerably more complicated than the first one. In particular, it has three more subcases depending on the structure of $G.$ Observe that with finishing the proof of the first case we will in fact finish proving Theorem~\ref{stab} as well, and in the second case we will have Theorem~\ref{main}.

Denote the set of vertices of $H$ belonging to a triangle by $V_{\Delta'}$, and the set of vertices belonging to a triangle containing only vertices having exactly 2 neighbors by 
$V_\Delta.$ Clearly, $V_\Delta\subset V_\Delta',$ the subgraph $H[V_\Delta]$ consists of a triangle factor, moreover, if $x \in V_\Delta,$  $y\in V(H),$ and $y$ do not belong to the triangle of $x$ then $xy\not\in E(H).$  

\begin{lemma}\label{vdelta}
If $H$ is $\nu$-triangular extreme, then $|V_\Delta|\ge n(1-7\nu).$
\end{lemma}
\begin{proof}
By the definition of triangular extremality we have $$|V_{\Delta'}|\ge n(1-\nu).$$ If $v\in V_{\Delta'}$ and $\deg_H(v)=3$ then it must be adjacent to a vertex in $V(H)-V_{\Delta'}$, moreover, a vertex in $V(H)-V_{\Delta'}$ is adjacent to at most 2 vertices in $V_{\Delta'}$, so there are at most $2\nu n$ vertices in  $V_{\Delta'}$ with degree 3, therefore 
\begin{eqnarray*}
|V_\Delta|&\ge& n(1-\nu-3\cdot 2\nu)\\
&=&n(1-7\nu).
\end{eqnarray*}
\end{proof}

\subsection{Non-extremal host graphs}\label{6.1}

In this section we assume that $G$ is not $\eta$-extremal with $\eta=(8\nu)^{1/1000}+126\nu,$ $\delta(G)\ge (2/3-\gamma)n,$ 
and $H$ is $\nu$-triangular extreme (recall, that $\nu=\sqrt[4]{\gamma}$).
%Recall the definition of $V_\Delta\subset V(H).$ 
Let $H'=H-V_\Delta.$ By Lemma~\ref{vdelta} we have that $v(H')\le 7\nu n.$

We give an itemized list of the embedding algorithm in this case.

\begin{itemize}

\item[] {\bf Step 1.}  Add $n-v(H')$ {\it fictive} isolated vertices to $H',$ and apply Proposition~\ref{harmadik} in order to embed this new graph into $G.$ 
Denote $\widehat{G}$ the subgraph of $G$ which is spanned by those vertices that are covered by the fictive vertices. Clearly, if $\widehat{G}$ contains 
$v(\widehat{G})/3$ vertex disjoint triangles, then we are done, so we will focus on finding these triangles in the sequel.\\

\item[]{\bf Step 2.} Let us first estimate the minimum degree of $\widehat{G}$ from below:
$$\delta(\widehat{G})\ge \delta(G)-v(H')\ge (2/3-\gamma-7\nu)n\ge (2/3-8\nu)v(\widehat{G}).$$ 
Next we show that $\widehat{G}$ is not $\eta'$-extremal for $\eta'= \eta-7\nu.$ Let $A \subset V(\widehat{G})$ with $|A|=v(\widehat{G})/3.$ Then $$e(\widehat{G}[A])\ge e(G[A])-7\nu n^2\ge (\eta/18-7\nu)n^2\ge  \eta' v(\widehat{G})^2/18.$$

\item[]{\bf Step 3.} By the above we obtained that $\delta(\widehat{G})\ge (2/3-8\nu)v(\widehat{G})$ and 
$\widehat{G}$ is $\eta'$-non-extremal, where $\eta'= (8\nu)^{1/1000}.$ Hence we may apply Theorem~\ref{Endre} in order to find the desired triangle factor in $\widehat{G}.$   

\end{itemize}

With this we finished proving the following.

\begin{lemma}\label{H-extrem}
Let $H$ be a $\nu$-triangular extreme graph on $n$ vertices with $\theta(H)\le 5,$ and let $G$ be an $\eta$-non-extremal graph on $n$ vertices with minimum degree 
$\delta(G)\ge (2/3-\gamma)n.$
If $n$ is sufficiently large then $H$ is a subgraph of $G.$
\end{lemma}

\subsection{Extremal host graphs}\label{6.2}

Our goal in this section is to prove the following.

\begin{lemma}\label{ExtremLemma}
There exists an $n_0$ such that if $n > n_0$, $\theta(H)\le 5$ for a $\nu$-triangular extreme graph $H$ of order $n$ and $\delta(G)\ge 2n/3$ for an $\eta$-extremal graph $G$ of order $n$ 
with $\eta\le (8\nu)^{1/1000}+126\nu$ then $H$ is a subgraph of $G.$
\end{lemma}

Note that with the above lemma we will finish the embedding of $H$ into an extremal $G.$ Before starting to prove it, we need to know more about the structure of $G.$ 
For that let us consider a $(\mu, 2)$-non-extremal 
%\footnote{Recall that we introduced non-extremality in Section~\ref{Review}} 
graph $G_1$ on $N$ vertices ($0<\mu<1$ is a constant) such that $N$ is even and 
$$\delta(G_1)\ge N/2-\alpha N$$ where $\alpha\le\mu/10.$ 
We are going to prove the following. 

\begin{lemma}\label{Bparositas} Let $G_1$ be as above. Then either $G_1$ has a perfect matching, or it has the following structure. Its vertex set can be divided 
into the disjoint subsets $V_1$ and $V_2$ such that $V(G_1)=V_1\cup V_2,$ $|V_1|=|V_2|=N/2,$ and $e(G_1[V_1, V_2])\le 3\alpha N^2.$ 
\end{lemma}

\begin{proof}
Assume that the largest matching $M$ in $G_1$ has less than $N/2$ edges, we will heavily use this assumption on the 
maximality of $M$ throughout the proof. We are going to prove that the required partition of $V(G_1)$ exists.

Given any vertex $a\in V(M),$ we let $M(a)$ denote its neighbor in $M,$ that is, $aM(a)\in M.$
Denote $u$ and $v$ any two vertices that are not covered by $M.$
We divide the set of edges of $M$ into six disjoint, not necessarily non-empty sets, based on how they are connected to $u$ and $v.$ 

\begin{itemize}

\item Let $M_1$ denote those $w_1w_2$ edges for
which $uw_1, uw_2\in E(G_1),$ that is, $V(M_1)\subset N_{G_1}(u).$ Since $M$ is maximal, $v$ has no neighbor in $V(M_1).$ 

\item Let $M_2$ denote the set of those $w_1w_2$ edges in which $u$ has exactly one neighbor, moreover, $v$ has no neighbor in $V(M_2).$

\item Next let $M_3$ denote the set of those edges in which both $u$ and $v$ has a neighbor. By the maximality of $M$ they can have exactly one neighbor in all the edges of $M_3.$ 
In fact, the maximality of $M$ implies that 
$N_{G_1}(u)\cap V(M_3)=N_{G_1}(v)\cap V(M_3).$ 

\item Let $M_4$ denote
the set of those edges of $M$ that has exactly one vertex that is adjacent to $v$ (and has no neighbor of $u$). 

\item Let $M_5$ denote the set 
of those edges of $M$ for which $V(M_5)\subset N_{G_1}(v)$ (hence no vertex of $V(M_5)$ is adjacent to $u$ by maximality). 

\item Finally, let $M_6$ denote the set of those edges for which $V(M_6)\cap (N_{G_1}(u)\cup N_{G_1}(v))=\emptyset.$

\end{itemize}

Let us set $m_i=|M_i|$ for every $1\le i \le 6,$ hence, $N\ge 2(m_1+\ldots +m_6),$ since there are unmatched points in $V(G_1).$ 
We decompose the proof of the lemma into several simple claims.

\begin{claim}\label{0-as}
Let $S_2=N_{G_1}(u)\cap V(M_2)$ and $S_4=N_{G_1}(v)\cap V(M_4).$ Then $G_1$ has no edges between $M(S_2)$ and $M(S_4).$
\end{claim}

\begin{proof}
If there were an edge between $M(S_2)$ and $M(S_4),$ we could find an augmenting path for $M,$ contradicting to the maximality of $M.$
\end{proof}

\begin{claim}\label{1-es}
We have $m_6\le \alpha N.$ 
\end{claim}
 
\begin{proof}
Using the definition of the $M_i$ sets and the fact that there are no edges between unmatched points we get that $$N/2 -\alpha N\le \deg_{G_1}(u)\le 2(m_1+m_2)+m_3$$ and  $$N/2 -\alpha N\le \deg_{G_1}(v)\le 2(m_4+m_5)+m_3.$$
Summing up the two inequalities we get that  $$N -2\alpha N\le 2(m_1+m_2+m_3+m_4+m_5)\le N-2m_6,$$ from which the claim follows.
\end{proof}

%\begin{claim}\label{2-es}
%We have $m_3\le N/2-\mu N.$ 
%\end{claim}

%\begin{proof}
%The claim easily follows from the observation that by maximality, the subset $S\subset V(M_3)$ that consists of the non-neighbors of $u$ and $v$
%must be an independent set, otherwise we could find an augmenting path for $M.$ Since $G_1$ is $(\mu, 2)$-non-extremal, we get the claimed bound.
%\end{proof}

%\begin{claim}\label{3-as}
%If $m_3>0$ then $m_1+m_5\le 2\alpha N.$ 
%\end{claim}
 
%\begin{proof}
%Let $w\in  M_3$ be an arbitrary vertex. Observe that $w$ does not have any neighbor in $V(M_1)\cup V(M_5)$ by the maximality of $M.$ 
%Then $$N/2-\alpha N\le \deg_{G_1}(w)\le N/2 -m_1-m_5+m_6\le N/2-m_1-m_5+\alpha N,$$ which easily implies
%the claimed bound.
%\end{proof}

\begin{claim}\label{3-as}
We have $M_3=\emptyset.$
\end{claim}

\begin{proof}
Assume that $M_3\not=\emptyset.$ First we show that  $m_3\le N/2-\mu N.$ By the maximality of $M,$ the subset $S\subset V(M_3)$ that consists of the non-neighbors of $u$ and $v$
must be an independent set, otherwise we could find an augmenting path for $M.$ Since $G_1$ is $(\mu, 2)$-non-extremal, simple calculations shows that $m_3$ cannot have more than $ N/2-\mu N$ edges.

Next let $w\in V(M_3)-N_{G_1}(u).$ By the maximality of $M$ we get that $w$ has {\it no neighbor} in the following subsets: $V(M_3)-N_{G_1}(u),$ $V(M_1)\cup V(M_5),$ $V(M_2)-N_{G_1}(u)$ and $V(M_4)-N_{G_1}(v).$ Since $w$ cannot have any unmatched neighbor, we get the following
$$\frac{N}{2}-\alpha N\le \deg_{G_1}(w)\le m_2+m_3+m_4+2m_6,$$ hence, $$\frac{N}{2}-3\alpha N\le m_2+m_3+m_4.$$
We proved that $m_3\le N/2-\mu N.$ Since $\alpha \le \mu/10,$ $w$ must have neighbors in $(V(M_2)\cap N_{G_1}(u)) \cup (V(M_4)\cap N_{G_1}(v)),$
denote $S$ this neighborhood. By the maximality of $M$ the set $M(S)\cup M(N_{G_1}(u)\cap V(M_3))$ must be an independent set. However, this set has at least $N/2-3\alpha N$ vertices, as we have just seen. This contradicts with the $(\mu, 2)$-non-extremality of $G_1.$ Hence
$M_3=\emptyset.$ 
 \end{proof}

\begin{claim}\label{4-es}
$m_2+m_4\le 2\alpha N.$
\end{claim}

\begin{proof}
Since $M_3=\emptyset,$ we get that $2m_1+m_2\ge N/2-\alpha N$ and $2m_5+m_4\ge N/2-\alpha N,$ here the first inequality follows from the degree bound of $u,$ the second follows from the degree bound of $v.$ We also have that $$N\ge 2(m_1+m_2+m_4+m_5)\ge N-2\alpha N.$$ Subtracting the first two inequalities from the last one we obtain what was desired. 
\end{proof}

Finally, we find the claimed partition of $V(G_1).$ First let $V_1=V(M_1)\cup V(M_2)\cup \{ u\}$ and $V_2=V(M_4)\cup V(M_5)\cup \{v\}.$
Observe that we can only have two unmatched points, $u$ and $v,$ otherwise for $u,v\ne w\not\in V(M)$ $$\deg_{G_1}(w)\le 2|M|-(2m_1+m_2+m_4+2m_5)= m_2+m_4+2m_6\le 4\alpha N <\delta(G_1).$$

Clearly $|V_i|> N/2-\alpha N$ ($i=1,2$).
If $|V_i|>N/2$ then move $|V_i|-N/2$ points from $V_i$ to $V_{3-i}$ ($i=1,2$). Denote the set of the moved points by $W.$
Then we distribute $V(M_6)$ among $V_1$ and $V_2$ so that the resulting sets have equal size. With this we have found the partition of $V(G_1).$ 

Next observe that if $v_1v_2\in E(G_1[V_1, V_2])$ (so it is a crossing edge), then either $v_1$ or $v_2$ belongs to $V(M_2)\cup V(M_4)\cup V(M_6)\cup W$
using the maximality of $M.$ Since the total number of vertices of this union of sets is at most $5\alpha N,$ we get that $e(G_1[V_1, V_2])\le \frac{5}{2}\alpha N^2<3\alpha N^2,$ what was desired.   
\end{proof}

\begin{proof}(of Lemma~\ref{ExtremLemma})
First we write $n$ in the form $n=3k+r,$ where $r\in \{0, 1, 2\}.$ 
Since $G$ is $\eta$-extremal, there exists a set $A\subset V(G)$ such that $|A|=k$ and $e(G[A])\le \eta n^2/18.$ Set $B=V(G)-A,$ and also set $\mu=10000\sqrt{\eta}.$ 

We consider three cases in the proof of the lemma. 
\begin{itemize}

\item [] {\bf Case 1:} $G[B]$ is $(\mu, 2)$-non-extremal, and for all $B_1\subset B,$ $|B_1|=k+\lfloor r/2\rfloor$ we have that 
$$e(G[B_1, B-B_1])\ge \mu n^2.$$

\item [] {\bf Case 2:} $G[B]$ is $(\mu, 2)$-non-extremal, and there exists a set $B_1\subset B,$ $|B_1|=k+\lfloor r/2\rfloor$ such that 
$$e(G[B_1, B-B_1])< \mu n^2.$$

\item [] {\bf Case 3:} $G[B]$ is $(\mu, 2)$-extremal. 
\end{itemize}

%It is easy to see that exactly one of the above cases must hold.
The proof of Lemma~\ref{ExtremLemma} is relatively simple when there are no  ``exceptional" vertices and $H$ consists of $n/3$ vertex disjoint triangles. For example, let us consider
Case 1 in such an ideal setup: $A$ is an independent set of size $k,$ $B$ has exactly $2k$ vertices, $G[A, B]=K_{k, 2k},$ $\delta(G[B])\ge k$ and $H$ has $n/3$ triangles. Then we can find a perfect
matching $M$ in $G[B]$ easily, since the minimum degree is sufficiently large in it. Next construct an auxiliary bipartite graph $F$ with vertex classes $A$ and $M.$ A vertex $u\in A$ is adjacent to $vw\in M$ in $F$ if and only if $uv, uw\in E(G).$
It is clear that $F$ has a perfect matching as $G[A, B]=K_{k, 2k}.$ Since every edge of $F$ translates to a triangle of $G,$ the existence of a perfect matching in $F$ implies the existence
of a triangle factor in $G,$ and this is what we wanted to prove.

However, $H$ may contain other structures, not only triangles, even though most of $H$ consists of vertex-disjoint triangles. Furthermore, our assumptions on the structure of $G$ are approximate. For example, in Case 1 a randomly chosen vertex of $A$ is expected to have almost all of its neighbors in $B.$
However, $A$ may contain some vertices that have perhaps only a bit 
more than $k$ neighbors in $B,$ and similarly, $B$ may also contain some vertices that have only a few neighbors in $A.$ In general, we will call a vertex {\it exceptional} whenever its neighborhood differs a lot from 
what we expect in case we choose a vertex randomly, according to the structure of $G.$ For all the three cases above we give rigorous definitions for the exceptional sets. 

We try to avoid having unnecessary complications, and therefore apply preprocessing methods that may reduce the number of exceptional vertices, and more importantly after preprocessing we can have structural
assumptions on the distribution of exceptional vertices. One of the preprocessing algorithms follows below,
it is needed in all the three cases. 

\medskip

\noindent {\bf Preprocessing \#1:}

\smallskip

If $A$ contains a vertex $u$ and $B$ contains a vertex $v$ such that $$\deg(u, A)+\deg(v, B)>\deg(u, B)+\deg(v, A),$$
then we switch the two vertices, that is, we let $A=A-u+v$ and $B=B-v+u.$ We stop when such vertices cannot be found.
Since in every step the number of edges between $A$ and $B$ increases, Preprocessing \#1 stops in a finite number of steps. 

Let us define two subsets of exceptional vertices:

$$A'=\left\{v\in A : \deg_G(v,B)< \frac{4n}{9}\right\},$$

$$B'=\left\{v\in B : \deg_G(v,A)< \frac{2n}{9}\right\}.$$

\begin{claim}\label{A'B'}
Apply Preprocessing \#1. If $A'\not=\emptyset$ then $B'=\emptyset.$ 
%Similarly, if $B'\not=\emptyset$ then $A'=\emptyset.$
\end{claim}

\begin{proof}
Simple calculation shows the claim using the definition of $A'$ and $B'.$  
\end{proof}

Of course this also means that if $B'\not=\emptyset$ then $A'=\emptyset.$
We also need the following definitions of other kind of exceptional vertices:
$$A''=\left\{v\in A-A' : \deg_G(v,B)\le \frac{2n}{3}-10\eta^{1/2}n\right\},$$

$$B''=\left\{v\in B-B' : \deg_G(v,A)\le \frac{n}{3}-10\eta^{1/2}n\right\}.$$

One may say that the vertices in $A''\cup B''$ are ``less exceptional" but still need some extra attention. 

\begin{claim}\label{meretek}
We have the following upper bounds: $|A'|\le \eta n,$ $|B'|\le 3\eta n,$ $|A''|\le \eta^{1/2}n/45$ and  $|B''|\le  \eta^{1/2}n/30.$
\end{claim}

\begin{proof}
Using that $G$ is $\eta$-extremal, the sum of the degrees inside $A$ is at most $\eta n^2/9,$ implying that $e(G[A, B])\ge (1-\eta) \frac{2n^2}{9}.$
\begin{enumerate}
\item $e(G[A, B])\le \left(\frac{n}{3}-|A'|\right)\frac{2n}{3}+|A'|\cdot\frac{4n}{9}=\frac{2n^2}{9}-\frac{2n}{9}|A'|,$ 
implying that $|A'|\le \eta n.$

\item $e(G[A, B])\le \left(\frac{2n}{3}+2-|B'|\right)\frac{n}{3}+|B'|\cdot\frac{2n}{9}=\frac{2n^2}{9}+\frac{2n}{3}-\frac{n}{9}|B'|,$ therefore $|B'|\le3\eta n.$

\item $e(G[A, B])\le \left(\frac{n}{3}-|A''|\right)\frac{2n}{3}+|A''|\left(\frac{2n}{3}-10\eta^{1/2}n\right)=\frac{2n^2}{9}-10\eta^{1/2}n|A''|,$ therefore $\eta^{1/2}n|A''|\le\frac{2\eta n^2}{90}$ implying that $|A''|\le\eta^{1/2}n/45.$

\item $e(G[A, B])\le \left(\frac{2n}{3}+2-|B''|\right)\frac{n}{3}+|B''|\left(\frac{n}{3}-10\eta^{1/2}n\right)=\frac{2n^2}{9}+\frac{2n}{3}-10\eta^{1/2}n|B''|,$
therefore $\eta^{1/2}n|B''|\le\frac{3\eta n^2}{90}$, implying that $|B''|\le  \eta^{1/2}n/30.$
\end{enumerate}

%The claim follows easily using this upper bound and calculating degrees for the four sets in question. 
\end{proof}

The following simple claim plays a key role in the proof of Lemma~\ref{ExtremLemma}.

\begin{claim}\label{B'parositas}
Assume that we are after applying Preprocessing \#1. If $B'\not=\emptyset$ then the bipartite subgraph $G[B', A]$ has a matching that covers every vertex of $B'.$ 
\end{claim}

\begin{proof} 
By the K\"onig-Hall theorem it is sufficient to show that for every $S\subset B'$ we have $$|N(S)\cap A|\ge |S|.$$
Let us assume that there exists an $S\subset B'$ for which the claimed inequality does not hold. Denote $s$ the cardinality of $S.$ Let $v\in S$ and $u\in A-N(S).$ Then $\deg(v, A)\le s-1,$ and 
$\deg(u, B)\le 2n/3-s.$ Let $A=A-u+v$ and $B=B-v+u.$ Clearly, this switching will increase the number of edges between $A$ and $B,$ contradicting to the Preprocessing algorithm \#1.
\end{proof}

\subsubsection{Case 1:}

Let us call a triangle $uvw$ {\it balanced} if exactly one of its vertices belongs to $A.$ Recall that $H'$ denotes the ''non-triangle'' part of $H.$ For embedding $H$ into $G$ we apply an algorithm that has several steps. 
As this algorithm
proceeds, more and more vertices will get covered in $A$ and in $B.$ It is useful to introduce a notation for the vacant subset of $A$ and $B$ after Step $j$ ($1\le j\le 4$): these will be denoted by $A(j)$ and $B(j).$
The embedding algorithm is as follows.

\medskip

\noindent {\bf Embedding Algorithm for Case 1:}

\smallskip

\begin{enumerate}

\item [Step 0:] Apply Preprocessing \#1.

\item [Step 1:] If $B'\not=\emptyset$: Cover the vertices of $B'$ by vertex-disjoint balanced triangles so that every such triangle consists of a vertex of $B',$ a vertex of $B-B',$ and a vertex of $A.$ If $A'\not=\emptyset$: Cover the vertices of $A'$ by vertex-disjoint balanced triangles. 

\item [Step 2:] Cover the vertices of $A(1)''$ and $B(1)''$ by vertex-disjoint balanced triangles, so that these triangles contain exactly one exceptional vertex (that is, the other two vertices of these triangles belong to $A(1)\cup B(1)$).

\item [Step 3:] Embed $H'$ so that its independent set $I(H')$ provided by Lemma~\ref{strH} is embedded into $A(2)$ and $H-I(H')$ is embedded into $B(2).$ 

\item [Step 4:] If necessary embed (less than $\lceil v(H')/3 \rceil$) vertex-disjoint triangles into $B(3)$ so as to get that at the end of this step $2|A(4)|=|B(4)|.$ 

\item [Step 5:] Find a triangle-factor in $G[A(4)\cup B(4)].$

\end{enumerate}

%We prove the correctness of the above algorithm in the following lemma.

\begin{lemma}\label{Case1}
If the conditions of Lemma~\ref{ExtremLemma} and Case 1 are satisfied, then the Embedding Algorithm for Case 1 finds a copy of $H$ in $G.$ 
\end{lemma}

\begin{proof}
We begin with Step 1. Assume that $B'$ is non-empty.
By Claim~\ref{B'parositas} we have a matching $M_1$ in the bipartite subgraph $G[A, B']$ that covers every vertex of $B'.$ Let $b\in B'$ and $a\in A$ such that $ab\in M_1.$ Since $b\in B'$ we have 
$\deg(b, B)\ge 4n/9.$ Similarly, $\deg(a, B)\ge 4n/9,$ since we must have $A'=\emptyset$ in this case. Therefore $|N(a)\cap N(b)\cap B|\ge 2n/9.$ By Claim~\ref{meretek} $|B'|\le 3\eta n,$ hence, for every
$ab\in M_1$ we can find a $v\in N(a)\cap N(b)\cap (B-B')$ such that the $abv$ triangles we obtain this way are vertex-disjoint. It is clear that these triangles are balanced, too.  

Let's assume now that $A'$ is non-empty. Since no vertex can have more than $n/3$ neighbors in $A,$ we have that $\deg(a, B)> n/3$ for every $a\in A'.$ Recall that $G[B]$ is $(\mu, 2)$-non-extremal, hence, the neighborhood $N(a)\cap B$ contains at least $\mu n^2/18$ edges. Pick one such edge with uncovered endpoints, it gives a triangle with $a.$ Repeat the above procedure until $A'$ is exhausted. Since $|A'|\le \eta n$
and $\eta \ll \mu,$ it is clear that we can find vertex-disjoint balanced triangles for $A'.$ Note that in every subset $S\subset B(1)$ with $|S|=k$ we still have at least $\mu n^2/18-k\cdot 10\eta n>\mu n^2/18-4\eta n^2$ edges 
in $G[S]$ with uncovered endpoints.

We continue with Step 2. For the vertices of $A(1)''$ we can use the method of Step 1 above, without any change.
 Next we consider the vertices of $B(1)''.$ Let $b\in B(1)''.$ We have that $\deg(b, A)\ge 2n/9$ and $\deg(b, B)\ge n/3.$ Pick a vacant point $a\in N(b, A(1))-A(1)'',$ there are more than $2n/9-\sqrt{\eta}n\gg |B(1)''|$
such vertices. Since $a\in A(1)-A(1)'',$ we have $|N(a)\cap N(b)\cap B(1)|\ge k-10 \sqrt{\eta}n \gg |B(1)''|,$ hence, for every $b\in B(1)''$ we can find an $a\in A(1)-A(1)''$ and a $v\in N(a)\cap N(b)\cap B(1)$ such that
the $abv$ triangles are vertex-disjoint and balanced. 

 Since $|A(1)''|+|B(1)''|\le  \sqrt{\eta}n/10$ there will be more than $\frac{\mu n^2}{18}-k\cdot 10\eta n-k\cdot \sqrt{\eta}n>(\mu/18-\sqrt{\eta}) n^2$ edges with uncovered endpoints in $G[S]$ 
for every subset $S\subset B$ with $|S|=k.$

Let us consider Step 3. 
Apply Lemma~\ref{strH} to $H'.$ Denote the resulting independent set by $I(H').$ By Lemma~\ref{strH} we know that the components of $H'-I(H')$ are paths having length at most 2. Let us add $25\sqrt{\eta} n$ {\it fictive points} to $B(2),$ and connect them with every vertex of $B(2).$ It is easy to see that the minimum degree in the graph we obtain this way is large enough to find a Hamilton cycle. We delete the fictive vertices.
Then the Hamilton cycle is divided into shorter paths having average length $1/(25\sqrt{\eta}).$ 

Since $v(H')\le 7\nu n$ and $\nu \ll \sqrt{\eta},$ we can apply a greedy procedure to map the vertices in $V(H')-I(H')$ using the collection of paths we obtained above. Next we have to map the vertices 
of $I(H').$ Since every
vertex of $B(2)$ has more than $n/3-2\sqrt{\eta}n$ neighbors in $A(2),$ we can again use a greedy algorithm. For example, let $x\in I(H')$ so that $y_1, \ldots, y_t$ are the neighbors of $x$ ($1\le t\le 4$). 
Suppose that 
$y_i$ was mapped onto $v_i\in B(2)$ for every $1\le i\le t.$ Since only non-exceptional vertices are left, the $v_i$s has at least $n/3-4\sqrt{\eta}n$ common neighbors  in $A(2)$, out of these at most $7\nu n$ are covered (by some other vertex of $I(H')$). Hence, we can always find
a vacant vertex $u\in A(2)$ for every $x$ so that adjacent vertices of $H'$ are mapped onto adjacent vertices of $G.$ With this we are done with Step 3.

For Step 4 we introduce some notation. Let $h_0=|I(H')|$ and $h_1=v(H')-h_0.$ Recall that $n=3k+r.$ Then we have $v(H')=3p +r$ for some $p\in \mathbb{N}.$ By Lemma~\ref{strH} $3h_0\ge h_0+h_1,$ 
hence, $2h_0=h_1+s$ for some $0\le s\le 7\nu n$ natural number. Set $k'=|A(2)|,$ then $|B(2)|=2k'+r.$ When embedding $H'$ we covered $h_1$ vertices of $B(2)$ and $h_0$ vertices of $A(2).$ 
The $A(3), B(3)$ sets are
{\it balanced}, if $2|A(3)|=|B(3)|.$ We claim that 
$B(3)$ contains $s+r$ vertices more than what it should be for being balanced: $$|B(3)|-2|A(3)|=2k'+r-h_1-2(k'-h_0)=r-h_1+2h_0=r+s.$$ 

We show that $r+s$ is divisible by three: $h_1=3p+r-h_0,$ hence, $h_1+s=3p+r-h_0+s,$ and using that $2h_0=h_1+s,$ we get that $3h_0=3p+r+s,$ implying the claimed divisibility.
Now our task is to embed $(r+s)/3$ triangles into $G[B(3)].$ This can be done using the $(\mu, 2)$-non-extremality of $G[B].$ Since we have covered less than $\mu n/20$ vertices of $B,$
$G[B(3)]$ is $(\mu/2 , 2)$-non-extremal. Hence, the neighborhood of any $b\in B(3)$  contains many edges, so we can find the $(r+s)/3$ vertex-disjoint triangles in $B(3)$ that are needed for the
balance. At the end we will have that $2|A(4)|=|B(4)|.$ Note that $G[B(4)]$ is $(\mu/3, 2)$-non-extremal, since $r+s\le 10 \nu n\ll \sqrt{\eta}n.$  

Finally, we finish the proof with Step 5.
By Lemma~\ref{Bparositas} and the fact that $G[B(4)]$ is $(\mu/3, 2)$-non-extremal and $\delta(G[B(4)])\ge (1/2-8\sqrt{\eta})|B(4)|,$ we can find a perfect matching $M$ in $G[B(4)].$

In order to extend $M$ into a triangle factor we consider the following auxiliary bipartite graph: its vertex classes are the vertices of $A(4)$ and the edges of $M,$ and we have an edge between $a\in A(4)$ and $b_1b_2\in M$ if and only
if $ab_1, ab_2\in E(G).$ Using the K\"onig-Hall theorem, we can find a perfect matching in this graph, as there are no exceptional vertices left: every edge of $M$ is adjacent to almost all of 
$A(4),$ and every vertex of $A(4)$ is adjacent to almost every edge of $M.$ 
Observing that a perfect matching in the auxiliary graph gives a triangle factor in $G[A(4) \cup B(4)]$ finishes the proof.
\end{proof}

\subsubsection{Case 2:}

In order to show that $H\subset G$ in Case 2 we will give the details of an embedding algorithm, which in many steps is similar to Embedding Algorithm 1, but there will be differences as well.
For example, we need Preprocessing \#1 discussed earlier, but we also need another version. 

Recall that in Case 2 the set $B$ can be divided into two subsets, $B_1$ and $B_2$ such that the number of edges between these two sets is less than $\mu n^2.$
Hence, in this case $G[B]$ is close to the union of two cliques having sizes $|B_1|=\lfloor |B|/2 \rfloor$ and $|B_2|=\lceil |B|/2 \rceil.$ 

We need the following:

\medskip

\noindent {\bf Preprocessing \#2:}

\smallskip

If $B_1$ has a vertex $v_1$ and $B_2$ has a vertex $v_2$ such that $$\deg(v_1, B_2)+\deg(v_2, B_1)>\deg(v_1, B_1)+\deg(v_2, B_2),$$ then we switch $v_1$ and $v_2,$ that is, we
let $B_1=B_1-v_1+v_2$ and $B_2=B_2-v_2+v_1.$ We stop when two such vertices cannot be found. Since every switching step decreases the number of edges between $B_1$ and $B_2,$
the algorithm stops in a finite number of steps. 

Let us define two subsets:

$$B_1'=\left\{v\in B_1 : \deg(v,B_1)< \frac{k}{3}\right\},$$

$$B_2'=\left\{v\in B_2 : \deg(v,B_2)< \frac{k}{3}\right\}.$$

\begin{claim}\label{Preproc2}
Apply Preprocessing \#1 and Preprocessing \#2. If $B_1'\not=\emptyset$ then $B_2'=\emptyset.$ Moreover, $B'\cap B_1'=B'\cap B_2'=\emptyset.$ %Similarly, if $B'\not=\emptyset$ then $A'=\emptyset.$
\end{claim}

\begin{proof}
Simple calculation shows the claim using the definitions of $B',$ $B_1'$ and $B_2'.$  
\end{proof}

We also need the following definitions:
$$B_1''=\left\{v\in B_1-B_1' : \deg(v, B_1)\le k-10\mu^{1/2}k\right\},$$

$$B_2''=\left\{v\in B_2-B_2' : \deg(v, B_2)\le k-10\mu^{1/2}k\right\}.$$

Recall the definition of $A', A'', B'$ and $B''.$ The following upper bounds hold.

\begin{claim}\label{meretek-2}
We have that $|A'|\le \eta n,$ $|B'|\le 3\eta n,$ $|A''|\le \eta^{1/2}n/45,$ $|B''| \eta^{1/2}n/30$, $|B_1'|, |B_2'|\le  5\mu n,$ and  $|B_1''|, |B_2''|\le  \mu^{1/2}n/2.$
\end{claim}

\begin{proof}
The proofs of the first four inequalities go in the same way as in the proof of Claim~\ref{meretek}.
For the other inequalities we use that $e(G[B_1, B_2]) \le \mu n^2.$
\begin{enumerate}
\item $e(G[B_1, B_2])\ge |B_1'|\left(\frac{n}{3}-\frac{k}{3}\right)\ge |B_1'|\frac{2n}{9},$ hence $ |B_1'|\le  5\mu n.$
Essentially the same computation shows that  $|B_2'|\le  5\mu n.$
 
\item $e(G[B_1, B_2])\ge |B_1''|10\sqrt{\mu}k> |B_1''|4\sqrt{\mu}n,$ implying that $|B_1''|\le  \mu^{1/2}n/2.$ 
This computation  can be repeated in order to obtain that  $|B_2''|\le  \mu^{1/2}n/2.$
\end{enumerate}
\end{proof}

\medskip

The following observation is an easy consequence of the minimum degree of $G,$ using that $|A|=k, |B_1|=k+\lfloor r/2\rfloor$ and $|B_2|=k+\lceil r/2 \rceil.$ 

\begin{observation}\label{atmeno}
Every $v\in B_i$ has a neighbor in $B_{3-i},$ where $i\in \{1, 2\}.$
\end{observation}

\medskip

As in the previous case, the embedding algorithm we use consists of several steps, during which more and more vertices of $G$ will be covered. Our notation will reflect this: for every $j\ge 1$ the vacant subset
of $A$ {\it after finishing} Step $j$ will be denoted by $A(j).$ The sets $B_1(j), B_2(j)$ are defined analogously. %A vertex, as before, can be covered, vacant or marked. Only vacant vertices can be marked.

\medskip

\noindent{\bf Embedding Algorithm for Case \#2} 

\begin{enumerate}

\item [Step 0:] Apply Preprocessing \#1 and Preprocessing \#2. Determine the sets of exceptional vertices. 

\item [Step 1:] Assuming that $B'$ is non-empty find the matching $M_{B'}$ in the bipartite subgraph $G[A, B']$ that covers every vertex of $B'$ using Claim~\ref{B'parositas}.  

\item [Step 2:] Cover the vertices of $B'(1), B_1'(1)$ and $B_2'(1)$ by vertex-disjoint balanced triangles such that $|B_2(2)|$ is even when finishing this step. Use the algorithm described in Lemma~\ref{specialis}. 

\item [Step 3:] Cover the vertices of $A'(2), A''(2)$ and $B''(2), B_1''(2), B_2''(2)$ by vertex-disjoint balanced triangles.

\item [Step 4:]  Embed $H'$ so that its independent set $I(H')$ provided by Lemma~\ref{strH} is embedded into $A(3)$ and $H'-I(H')$ is embedded into $G[B_1(3)].$  If necessary embed 
(at most $\lceil v(H')/3 \rceil$) vertex-disjoint triangles into $B_1(3)$ so as to get that $2|A(4)|=|B_1(4)|+|B_2(4)|.$

\item [Step 5:] Find a triangle-factor in $G[A(4)\cup B(4)].$

\end{enumerate}

Call a balanced triangle $uvw$ {\it crossing,} if $v\in B_1, w\in B_2$ and $u\in A,$ otherwise it is called {\it non-crossing.} So a non-crossing triangle contains one vertex from $A$ and two vertices that either 
both belong to $B_1$ or to $B_2.$ Before we present the algorithm required in Step 2, we need a claim.

\begin{claim}\label{B'paritas}
Assume that $b=|B'|>0.$ Let $q$ be a natural number with $0\le q\le b.$ Then we can cover the vertices of $B'$ by vertex-disjoint balanced triangles such that exactly $q$ of these are crossing, and $b-q$ are non-crossing. 
\end{claim}

\begin{proof}
By Claim~\ref{B'parositas} we have a matching $M_{B'}$ in $G[A, B']$ that covers $B'.$ Let $S\subset B'$ be an arbitrary subset with $|S|=q,$ we are going to find crossing triangles for the vertices of $S,$ and non-crossing ones for $B-S.$ 

Let $v\in S$ be arbitrary. By the definition of $B'$ we know that $v$ has at least $k/3$ neighbors in the opposite $B_i$ set (in $B_2(1),$ if $v\in B_1(1),$ else in $B_1(1)$), denote this neighborhood by $N_v$. 
If $\deg(v, A)\ge k/100$ then $v$ has more than $k/200$ such neighbors $u_v\in A$ that each has at least $k/4$ neighbors in $N_v.$ Hence, in this case the crossing triangle will be $vwu_v,$ where $w\in N_v$
is adjacent to $u_v.$ If $\deg(v, A)< k/100$ then $|N_v|\ge 99k/100.$ Since $A'=\emptyset$ (by Claim~\ref{A'B'}) we get that $M_{B'}(v)$ has at least $k/4$ neighbors in $N_v,$ so we can pick a vacant vertex $w$ among them. The crossing triangle in this case is $vwM_{B'}(v).$  

For $v\in B-S$ we can essentially repeat the above argument. The only difference is that this time we look for a $w$ not in the opposite $B_i$ set, but in the set of $v,$ we leave the details for the reader. 
\end{proof}

\begin{lemma}\label{specialis}
It is possible to cover the vertices of $B'(1), B_1'(1)$ and $B_2'(1)$ by vertex-disjoint balanced triangles such that $|B_2(2)|$ is even when finishing Step 2.~of Embedding Algorithm for Case \#2. 
\end{lemma}

\begin{proof}
We have two main cases, depending on whether $|B_2(1)|$ is odd or even.  
%We remark that in the discussion below we assume for simplicity that $B_2'()=\emptyset.$ The reader 
%may notice that essentially the same method works in case $B_2'()\not=\emptyset,$ only straightforward modifications are needed.
Assume first that $|B_2(1)|$ is odd and $B_2'(1)=\emptyset.$ We consider several sub-cases.

%In Claim we proved that after embedding $H'$ we have that $2|A()|=|B()|,$ hence, $|B()|=|B_1()|+|B_2|$ must be even (here we used that before Step no vertex of $B_2$ gets %covered, hence, $B_2()=B_2$). Therefore $|B_1()|$ is odd.

\begin{itemize}

\item[(i)] $|B_1'(1)|$ is odd: We will find crossing triangles for $B_1'(1).$ Since every vertex of $B_1'(1)$ has at least $2k/3$ neighbors in $B_2,$ we can easily find distinct neighbors for each of them. By Claim~\ref{Preproc2} each $v\in B_1'(1)$ has more than 
$k/2$ neighbors in $A(1)-V(M_{B'}).$ The vast majority of these neighbors have almost full degree to $B_2(1),$ hence, we can find the claimed vertex disjoint balanced triangles that cover $B_1'(1).$ 

Since $|B_2(1)|$ was odd
and we used up an odd number of vertices from it, at this point $B_2(1)$ has an even number of vacant vertices. Then we cover $B'(1)$ so that we use an even number of vacant vertices from $B_2(1)$ by Claim~\ref{B'paritas}, just find non-crossing triangles for every vertex of $B'.$

\item [(ii)] $|B_1'(1)|$ is even and $B'(1)\not=\emptyset$: First, as above we find crossing triangles for $B_1'(1).$ Then apply Claim~\ref{B'paritas} with $q=1,$ find $1$ crossing triangle for covering a vertex of $B'(1),$
and non-crossing ones for other vertices of $B'(1).$ Clearly, we use up an odd number of vertices from $B_2(1)$ this way.

\item [(iii)] $|B_1'(1)|$ is even and $B'(1)=\emptyset$: Let $v\in B_1(1)-B_1'(1)-B_1''(1)$ be an arbitrary vertex. By Observation~\ref{atmeno} $v$ must has at least one neighbor 
$w \in B_2=B_2(1).$ Moreover,
$|N(v)\cap A(1)\cap N(w)|>k/4$ since $v$ (being non-exceptional) has almost full degree into $A,$ and $\deg(w, A)\ge k/3$ since $B'=\emptyset.$ Let $u\in N(v)\cap A(1)\cap N(w)$ be an arbitrary vacant vertex. 
The $uvw$ triangle is a non-crossing one. Next we find vertex-disjoint crossing triangles for $B_1'(1),$ as is described above. Observe, that with this we have used up an odd number of vertices of $B_2(1),$ 
hence, after embedding it the vacant part of $B_2$ will have even cardinality. 
\end{itemize}

If $|B_2(1)|$ is odd and $B_2'(1)\not=\emptyset,$ then only minor modifications are needed. The details are as follows.

\begin{itemize}

\item[(i)] $|B_2'(1)|$ is odd: Repeat the method of the first case above, with $B_2'(1)$ instead of $B_1'(1).$

\item [(ii)] $|B_2'(1)|$ is even and $B'(1)\not=\emptyset$: Again, repeat the method of the second case above, with $B_2'(1)$ instead of $B_1'(1).$

\item [(iii)] $|B_2'(1)|$ is even and $B'(1)=\emptyset$: By Observation~\ref{atmeno} every vertex $v\in B_2$ has at least one neighbor in $B_1.$ Pick a vertex 
$v\in B_2(1)-B_2'(1)-B_2''(1),$ then find a $w\in B_1(1)$ and $u\in A(1)$ such that $uvw$ is a crossing triangle. Next find vertex-disjoint crossing triangles for $B_2'(1).$ 
 
\end{itemize}

Our second main case is when $|B_2(1)|$ is even. We don't want to give the details of the whole argument, as it is very similar to the case when $|B_2(1)|$ is odd, we only give an outline. First, regardless of whether 
$B_1'(1)$ or $B_2'(1)$ is non-empty, we find vertex-disjoint crossing triangles for them. Doing so may change the parity of the vacant subset of $B_2$ from even to odd. When $B'(1)\not=\emptyset,$ then
using Claim~\ref{B'paritas} we can easily find zero or one crossing triangle, and non-crossing ones for the rest of $B'(1),$ in order to get an even $B_2(2)$ at the end of Step 2. And when $B'(1)=\emptyset,$ we use
the minimum degree condition of $G$ (see step (iii) above) in order to find a crossing triangle -- one can see that this is necessary only when $|B_1'(1)|$ or $|B_2'(1)|$ is odd.    
\end{proof}

Next we prove the correctness of the Embedding Algorithm for Case \#2.

\begin{lemma}\label{Case2}
If the conditions of Lemma~\ref{ExtremLemma} and Case 2 are satisfied, then the Embedding Algorithm for Case 2 finds a copy of $H$ in $G.$ 
\end{lemma}

\begin{proof}
Using Lemma~\ref{specialis} we can cover every vertex of $B'(1), B_1'(1)$ and $B_2'(1)$ by vertex-disjoint triangles, moreover, $B_2(2)$ will have even cardinality. 

In Step 3 we cover the rest of the exceptional vertices.  
For $A'(2), A''(2)$ and $B''(2)$ one can use the method discussed in Lemma~\ref{Case1} for Case 1. However, we have to be more careful: only non-crossing triangles can be used since we need that $|B_2(3)|$ is even. This is easily doable for the vertices of $A'(2), A''(2),$ since they have more than $k/4$ neighbors in $B_1(2)$ and in $B_2(2)$ as well. For covering $B''(2)$ observe that any $v\in B''(2)$ has more than $k/4$ neighbors in
its $B_i$ set, as $B_1'(2)\cap B''(2)=B_2'(2)\cap B''(2)=\emptyset.$ Hence, it is easy to find non-crossing vertex-disjoint triangles for them. 

For $B''_1(2)$ and $B_2''(2)$ we do the following. Say that $v\in B_1''(2),$ then 
both $\deg(v, B_1(2))$ and $\deg(v, A(2))\ge k/4 \gg |B_1''(2)|$ by Claim~\ref{meretek-2}. So we can pick a vertex $u\in (A(2)-A'(2)-A''(2))\cap N(v)$ that is adjacent to almost every vertex of $N(v)\cap B_1(2),$ hence, we have
many possibilities for a triangle that has two vertices in $B_1(2)$ and one vertex in $A(2).$ The same method can be applied for $B_2''(2),$ too.  

In Step 4 we can apply Lemma~\ref{Case1}, since the minimum degree in $G[B_1(3)]$ is larger than $k-10\sqrt{\mu}k.$ In Lemma~\ref{Case1} we also proved that after embedding $H'$ and $(r+s)/3$ vertex-disjoint triangles into $B$ the vacant part of $B$ has exactly twice as many vertices as the vacant part of $A.$ It is easy to see that this statement still holds, that is, $2|A(4)|=|B_1(4)|+|B_2(4)|.$ This implies
that $|B_1(4)|+|B_2(4)|$ is even. Since we managed to achieve that $|B_2(4)|$ is even, $|B_1(4)|$ must be even as well.

In order to show that one can find the desired triangle factor in $G[A(4)\cup B(4)]$ we do the following. Divide $A(4)$ into two disjoint subsets, $A_1$ and $A_2$ such that $2|A_1|=|B_1(4)|$ and $2|A_2|=|B_2(4)|.$ 
Then apply the method given in Lemma~\ref{Case1} for finding a triangle factor in $G[A_1 \cup B_1(4)]$ and $G[A_2\cup B_2(4)]$ separately. This is doable since every vertex of $A(4)$ has almost full degree into
$B_1(4)\cup B_2(4)$ and every vertex of $B_1(4)\cup B_2(4)$ has almost full degree into $A(4)$ as the exceptional vertices have all been covered. This finishes the proof of the lemma.   
\end{proof}

\subsubsection{Case 3:}

In this last case $V(G)$ can be partitioned into three disjoint sets, $A, B$ and $C$ such that $|A|=k,$ $|B|=k+\lfloor r/2 \rfloor,$ $|C|= k+\lceil r/2 \rceil$ (so $|A|\le |B|\le |C|$) and $e(G[A]), e(G[C]), e(G[B])\le \mu n^2/18.$ For proving that 
$H\subset G$ in Case 3 we use very similar methods to the previous ones. The main idea is again to first cover the exceptional vertices, however, for embedding the vast majority of $H$ we will use the Blow-up Lemma.
This makes the presentation of this case short. 

Let us introduce a slightly different kind of preprocessing.

\medskip

\noindent {\bf Preprocessing \#3:}

\smallskip

{\it Ordinary switching:} Let $X, Y$ denote two different partition sets of $V(G)$ (for example $X=A$ and $Y=C$). If there exists $u\in X$ and $v\in Y$ such that $$d(u, X)+d(v, Y)> d(u, Y)+d(v, X),$$ then we switch
$u$ and $v,$  that is, we let $X=X-u+v$ and $Y=Y-v+u.$ 

{\it Circular switching:} Let $X_1, X_2, X_3$ denote the partition sets of $V(G)$ in some order. For us there will be two different orders: in both of them $X_1=A$ and either $X_2=B$ and $X_3=C$ or $X_2=C$ and
$X_3=B.$ If there exists $u\in X_1,$ $v\in X_2$ and $w\in X_3$ such that $$\deg(u, X_1)+\deg(v, X_2)+\deg(w, X_3)> \deg(u, X_2)+\deg(v, X_3)+\deg(w, X_1),$$ then we do a circular switching, that is, we
let $X_1=X_1-u+w,$ $X_2=X_2-v+u$ and $X_3=X_3-w+v.$ 

We stop when such vertices cannot be found. Since every switching step increases the number of edges between the three partition sets, 
the algorithm stops in a finite number of steps. 

Let us define three exceptional sets:

$$A'=\left\{v\in A : \ {\rm either} \  \deg(v, B)< \frac{k}{3} \ {\rm or} \ \deg(v, C)< \frac{k}{3} \right\},$$

$$B'=\left\{v\in B : \ {\rm either} \  \deg(v, A)< \frac{k}{3} \ {\rm or} \ \deg(v, C)< \frac{k}{3} \right\},$$

$$C'=\left\{v\in C : \ {\rm either} \  \deg(v, A)< \frac{k}{3} \ {\rm or} \ \deg(v, B)< \frac{k}{3} \right\}.$$

\begin{claim}
Let $X, Y, Z$ denote the partition sets $A, B, C$ of $V(G)$ in some order. If $u\in X'$ then $u$ has at least $2k/3$ neighbors either in $Y$ or in $Z.$ 
\end{claim}

\begin{proof} Follows easily from the minimum degree bound for $G.$
\end{proof}

\begin{claim}\label{Preproc3}
After applying Preprocessing \#3 there could remain at most two exceptional sets out of the above three. Let $X, Y$ denote any two different partition sets of $V(G)$ such that $X$ has non-empty exceptional subset $X'.$  
If there exist a vertex $u\in X'$ such that $\deg(u, Y)<k/3,$ then every vertex of $Y$ has more than $2k/3$ neighbors in $X.$ 
\end{claim}

\begin{proof}
An easy observation implies the first statement of the claim: if $A', B'$ and $C'$ are all non-empty sets then one could do either an ordinary or a circular switching step. The second statement is essentially equivalent to Claim~\ref{A'B'}.
\end{proof}

\begin{claim}\label{Excparosito}
Let $X, Y, Z$ denote the partition sets $A, B, C$ of $V(G)$ in some order. Assume that $X'$ is non-empty, and let $X_Y\subset X'$ denote the set of those vertices that have less than $k/3$ neighbors in $Y,$ 
$X_Z\subset X'$ is defined analogously. Then there exists a matching $M_Y$ in the bipartite subgraph $G[X_Y, Y]$ that covers every vertex of $X_Y,$ and similarly, there exists a matching $M_Z$ in the bipartite subgraph $G[X_Z, Z]$ that covers every vertex of $X_Z.$ If $X', Y'$ are non-empty then exists a matching $M$ in the bipartite subgraph $G[X'\cup Y', Z]$ that covers every vertex of $X'\cup Y'.$
\end{claim}

\begin{proof}
The claim is easily seen to follow from Claim~\ref{B'parositas} and from Claim~\ref{Preproc3}.
\end{proof}

We also need the following definitions of exceptional sets:

$$A''=\left\{v\in A-A' : \ {\rm either} \ \deg(v, B)\le k-10\mu^{1/2}k \ {\rm or} \  \deg(v, C)\le k-10\mu^{1/2}k\right\},$$

$$B''=\left\{v\in B-B' : \ {\rm either} \ \deg(v, A)\le k-10\mu^{1/2}k \ {\rm or} \  \deg(v, C)\le k-10\mu^{1/2}k\right\},$$

$$C''=\left\{v\in C-C' : \ {\rm either} \ \deg(v, A)\le k-10\mu^{1/2}k \ {\rm or} \  \deg(v, B)\le k-10\mu^{1/2}k\right\}.$$

The following upper bounds hold.

\begin{claim}\label{meretek-3}
We have $|A'|, |B'|, |C'|\le \mu n/2,$ and $|A''|, |B''|, |C''| \le \mu^{1/2}n/30.$
\end{claim}

\begin{proof}
Let $X$ denote any of the $A, B, C$ sets and define $X'$ and $X''$ analogously. Recall that $e(G[X])\le \mu n^2/18$ in Case~3.

Since $2e(G[X])\ge |X'|\left(\frac{2n}{3}-\frac{k}{3}-(k+1)\right)\ge |X'|\frac{2n}{9}$ we get that $|X'| \le \mu n/2.$
Similarly, 
$2e(G[X])\ge |X''|\cdot 10\sqrt{\mu}k$ implying that $|X''|\le \sqrt{\mu}n/30.$
\end{proof}

As before, when we refer to sets $A(i), B(i)$ or $C(i)$ below, we refer to the vacant subsets of $A, B$ and $C,$ respectively, right after Step $i.$

\medskip

\noindent{\bf Embedding Algorithm for Case \#3} 

\begin{enumerate}

\item [Step 0:] Apply Preprocessing \#3. Determine the sets of exceptional vertices. 

\item [Step 1:] Cover the vertices of $A', B'$ and $C'$ by vertex-disjoint crossing triangles, using the matching of Claim~\ref{Excparosito}. 

\item [Step 2:] Cover the vertices of $A'', B''$ and $C''$ by vertex-disjoint crossing triangles.  

\item [Step 3:] Use Theorem~\ref{equicol} to find an {\it equitable coloring} of $H'.$ Denote the color classes of $H'$ by $K_1, K_2$ and $K_3,$ such that $|K_1|\le |K_2|\le |K_3|.$

\item [Step 4:] Use the Blow-up Lemma (Theorem~\ref{blow-up}) in order to embed $H'$ and the missing number of vertex-disjoint triangles.

\end{enumerate}

\begin{lemma}\label{Case3}
If the conditions of Lemma~\ref{ExtremLemma} and Case 3 are satisfied, then the Embedding Algorithm for Case 3 finds a copy of $H$ in $G.$ 
\end{lemma}

\begin{proof}
For proving the lemma it is sufficient to use the ideas developed for the previous two cases. By Claim~\ref{Excparosito} we have a matching that covers $A' \cup  B'\cup C'.$ Using this matching we can find a
set of vertex-disjoint crossing triangles as in Claim~\ref{B'paritas}.
Covering of the vertices of $A'' \cup B''\cup C''$ is also done similarly to previous cases. 

Before we apply the Blow-up Lemma first assign the color class $K_1$ to $A(3),$ $K_2$ to $B(3)$ and $K_3$ to $C(3).$ Observe that $|A(3)|-|K_1|=|B(3)|-|K_2|=|C(3)|-|K_3|.$ 
There are no exceptional vertices left in Step 4, and $$|A'\cup B'\cup C'\cup A''\cup B'' \cup C''|\le \sqrt{\mu}n,$$ hence, the bipartite subgraphs $G[A(3), B(3)],$ $G[A(3), C(3)]$ and $G[B(3), C(3)]$ 
are each $(\varepsilon, \delta)$-super-regular pairs with $\varepsilon \le \sqrt[4]{\mu}$ and $\delta\ge n/3-30\sqrt{\mu}n.$ Hence, embedding the rest of $H$ can be done by a routine application of the Blow-up Lemma.
\end{proof}

Putting together Lemma~\ref{Case1}, Lemma~\ref{Case2} and Lemma~\ref{Case3} we obtain what was desired: $H$ is a subgraph of $G$ when both graphs are extremal.
\end{proof}

\end{document}